\documentclass{amsproc}
\usepackage{amssymb}
\usepackage{mathtools}

\newtheorem{theorem}{Theorem}[section]
\newtheorem{lemma}[theorem]{Lemma}
\newtheorem{proposition}[theorem]{Proposition}

\theoremstyle{definition}

\theoremstyle{remark}
\newtheorem{remark}[theorem]{Remark}

\numberwithin{equation}{section}

\newcommand\R{\mathbb{R}}
\newcommand\C{\mathbb{C}}
\newcommand\E{\mathbb{E}}
\newcommand\nl{\operatorname{nlin}}
\newcommand\lin{\operatorname{lin}}

\begin{document}

\title[Quantitative bounds for Navier-Stokes]{Quantitative bounds for critically bounded solutions to the Navier-Stokes equations}

\author{Terence Tao}
\address{UCLA Department of Mathematics, Los Angeles, CA 90095-1555}
\email{tao@math.ucla.edu}
\thanks{The author is supported by NSF grant DMS-1266164 and by a Simons Investigator Award.  We also thank Stan Palasek and Jiayan Wu for corrections.}

\date{}



\begin{abstract}
We revisit the regularity theory of Escauriaza, Seregin, and \v{S}ver\'ak for solutions to the three-dimensional Navier-Stokes equations which are uniformly bounded in the critical $L^3_x(\R^3)$ norm.  By replacing all invocations of compactness methods in these arguments with quantitative substitutes, and similarly replacing unique continuation and backwards uniqueness estimates by their corresponding Carleman inequalities, we obtain quantitative bounds for higher regularity norms of these solutions in terms of the critical $L^3_x$ bound (with a dependence that is triple exponential in nature).  In particular, we show that as one approaches a finite blowup time $T_*$, the critical $L^3_x$ norm must blow up at a rate $(\log\log\log \frac{1}{T_*-t})^c$ or faster for an infinite sequence of times approaching $T_*$ and some absolute constant $c>0$.
\end{abstract}

\maketitle

\section{Introduction}

This paper is concerned with quantitative bounds for solutions $u: [0,T] \times \R^3 \to \R^3$, $p: [0,T] \times \R^3 \to \R$ to the Navier-Stokes equations
\begin{equation}\label{ns}
\begin{split}
\partial_t u + (u \cdot \nabla) u &= \Delta u - \nabla p\\
\nabla \cdot u &= 0.
\end{split}
\end{equation}
Here we have normalised the viscosity to equal one for simplicity.
To avoid technicalities, we shall restrict attention to \emph{classical} solutions, by which we mean solutions that are smooth and such that all derivatives of $u, p$ lie in the space $L^\infty_t L^2_x([0,T] \times \R^3)$.  As our bounds are quantitative and do not depend on any smooth norms of the solution, it is possible to extend the results here to weaker notions of solution, such as mild solutions of Kato \cite{kato}, the weak Leray-Hopf solutions studied in \cite{ess}, or the suitable weak solutions from \cite{ckn}, by using the regularity theory of such solutions; we leave the details to the interested reader.  As is well known, such solutions have a maximal Cauchy development $u: [0,T_*) \times \R^3 \to \R^3$, $p: [0,T_*) \times \R^3 \to \R$ for some $0 < T_* \leq \infty$, with the restriction to $[0,T] \times \R^3$ a classical solution for all $T < T_*$, but for which no smooth extension to time $T_*$ is possible if $T_* < \infty$.  We refer to $T_*$ as the \emph{maximal time of existence} of such a classical solution.

The Navier-Stokes system enjoys the scaling symmetry $(u,p,T) \mapsto (u^\lambda,p^\lambda,\lambda^2 T)$ for any $\lambda>0$, where
$$ u^\lambda(t,x) \coloneqq \lambda u(\lambda^2 t, \lambda x)$$
and
$$ p^\lambda(t,x) \coloneqq \lambda^2 p(\lambda^2 t, \lambda x),$$
Among other things, this means that the norm
$$ \| u \|_{L^\infty_t L^3_x([0,T] \times \R^3)}$$
is scale-invariant (or \emph{critical}) for this equation.  In \cite{ess} it was shown that as long as this norm stays bounded, solutions to Navier-Stokes remain regular.  In particular, they showed an endpoint of the classical Prodi-Serrin-Ladyshenskaya blowup criterion \cite{prodi}, \cite{serrin}, \cite{lady-0} or the Leray blowup criterion \cite{leray}:

\begin{theorem}[Qualitative blowup criterion]\label{opp} \cite{ess} Suppose $(u,p)$ is a classical solution to Navier-Stokes whose maximal time of existence $T_*$ is finite.  Then
$$ \limsup_{t \to T_*^+} \| u(t) \|_{L^3_x(\R^3)} = +\infty.$$
\end{theorem}

There are now many proofs, variants and generalisations \cite{ess}, \cite{kk}, \cite{gkk}, \cite{gkk2}, \cite{seregin}, \cite{phuc}, \cite{gip} \cite{dong}, \cite{albritton}, \cite{barker}, \cite{ss}, \cite{wz} of this theorem, including extensions to higher dimensions or other domains than Euclidean spaces, replacing $L^3$ with another critical Besov or Lorenz space, or replacing the limit superior by a limit.  However, in contrast to the more quantitative arguments of Leray, Prodi, Serrin and Ladyshenskaya, the proofs in the above references all rely at some point on a compactness argument to extract a limiting profile solution to which qualitative results such as unique continuation and backwards uniqueness for heat equations (as established in particular in \cite{ess-back}) can be applied.  As such, the above proofs do not easily give any quantitative rate of blowup for the $L^3$ norm.

On the other hand, the proofs of unique continuation and backwards uniqueness rely on explicit Carleman inequalities which are fully quantitative in nature.  Thus, one would expect it to be possible, at least in principle, to remove the reliance on compactness methods and obtain a quantitative version of Theorem \ref{opp}.  This is the purpose of the current paper.  More precisely, in Section \ref{final-sec} we will establish the following two results:

\begin{theorem}[Quantitative regularity for critically bounded solutions]\label{main-1}  Let $u: [0,T] \times \R^3 \to \R^3$, $p: [0,T] \times \R^3 \to \R$ be a classical solution to the Navier-Stokes equations with
\begin{equation}\label{u3}
 \| u \|_{L^\infty_t L^3_x([0,T] \times \R^3)} \leq A 
\end{equation}
for some $A \geq 2$.  Then we have the derivative bounds
$$ |\nabla^j_x u(t,x)| \leq \exp\exp\exp(A^{O(1)})  t^{-\frac{j+1}{2}}$$
and
$$ |\nabla^j_x \omega(t,x)| \leq \exp\exp\exp(A^{O(1)})  t^{-\frac{j+2}{2}}$$
whenever $0 < t \leq T$, $x \in \R^3$, and $j = 0,1$, where $\omega \coloneqq \nabla \times u$ is the vorticity field.  (See Section \ref{notation-sec} for the asymptotic notation used in this paper.)
\end{theorem}

\begin{remark}
It is not difficult to iterate using Schauder estimates in H\"older spaces and extend the above regularity bounds to higher values of $j$ than $j=0,1$ (allowing the implied constants in the $O()$ notation to depend on $j$), and also control time derivatives (conceding a factor of $t^{-1}$ for each time derivative); we leave this extension of Theorem \ref{main-1} to the interested reader.
\end{remark}

\begin{theorem}[Quantitative blowup criterion]\label{main-2}  Let $u: [0,T_*) \times \R^3 \to \R^3$, $p: [0,T_*) \times \R^3 \to \R$ be a classical solution to the Navier-Stokes equations which blows up at a finite time $0 < T_* < \infty$.  Then 
$$ \limsup_{t \to T_*^-} \frac{\|u(t)\|_{L^3_x(\R^3)}}{(\log\log\log \frac{1}{T_*-t})^c} = +\infty$$
for an absolute constant $c>0$.
\end{theorem}

We now discuss the method of proof of these theorems, which uses many of the same key inputs as in previous arguments (most notably the Carleman estimates used to prove backwards uniqueness and unique continuation), but also introduces some other ingredients in order to avoid having to make some rather delicate results from the qualitative theory (such as profile decompositions) quantitative, as doing so would almost certainly lead to much poorer bounds than the ones given here.  

The main estimate focuses on bounding the scale-invariant quantity
\begin{equation}\label{pntx}
 N_0^{-1} |P_{N_0} u(t_0,x_0)|
\end{equation}
for various points $(t_0,x_0)$ in spacetime, and various frequencies $N_0$, where $P_{N_0}$ is a Littlewood-Paley projection operator to frequencies $\sim N_0$ (see Section \ref{notation-sec} for a precise definition). Using \eqref{u3} and the Bernstein inequality, one can bound this quantity by $O( A )$.  It is well known that if one could improve this bound somewhat for sufficiently large $N_0$, for instance to $O(A^{-C_0} )$ for a large constant $C_0$, then (assuming $A$ is large enough) the $L^3$ norm becomes sufficiently ``dispersed'' in space and frequency that one could adapt the local well-posedness theory for the Navier-Stokes equation (or the local regularity theory from \cite{ckn}) to obtain good bounds.  Hence we will focus on establishing such a bound for \eqref{pntx} for $N_0$ large\footnote{Strictly speaking, it is the scale-invariant quantity $N_0^2 T$ that needs to be large, rather than $N_0$ itself, where $T$ is the amount of time to the past of $x_0$ for which the solution exists and obeys the bounds \eqref{u3}.} enough (see Theorem \ref{main-est} for a precise statement).

The first step in doing so is to observe (basically from the Duhamel formula and some standard Littlewood-Paley theory) that if the quantity \eqref{pntx} is large for some $N_0,t_0,x_0$ with $t_0$ not too close to the initial time $0$, then the quantity
\begin{equation}\label{ptnx-2}
N_1^{-1} |P_{N_1} u(t_1,x_1)|
\end{equation}
is also large (with exactly the same lower bound) for some $(t_1,x_1)$ a little bit to the past of $(t_0,x_0)$ (but more or less within the ``parabolic domain of dependence'', in the sense that $x_1 = x_0 + O( (t_0-t_1)^{1/2})$) and with $N_1$ comparable to $N_0$; see Proposition \ref{basic}(iv) for a precise statement.  If one takes care to have exactly the same lower bounds for both \eqref{pntx} and \eqref{ptnx-2}, then this claim can be iterated, creating a chain of ``bubbles of concentration'' at various points $(t_n,x_n)$ and frequencies $N_n$, propagating backwards in time, and for which
$$
N_n^{-1} |P_{N_n} u(t_n,x_n)|
$$
is bounded from below uniformly in $n$.  Furthermore, by using a ``bounded total speed'' property first observed in \cite{tao-local}, one can ensure that $(t_n,x_n)$ stays in the ``parabolic domain of dependence'' in the sense that $x_n= x_0 + O((t_0-t_n)^{1/2})$. Due to the well known fact (dating back to the classical work of Leray \cite{leray}) that solutions to Navier-Stokes enjoy large ``epochs of regularity'' in which one has control of high regularity norms of the solution in large time intervals outside of a small dimensional singular set of times (see Proposition \ref{basic}(iii) for a precise quantification of this statement), one can show that there are a large number of points $(t_n,x_n)$ for which the frequency $N_n$ is basically as small as possible, in the sense that
$$ N_n \sim |t_0-t_n|^{-1/2}.$$
The (Littlewood-Paley component $P_{N_n} u$ of) the solution $u$ is large near $(t_n,x_n)$, and it is not difficult to then obtain analogous lower bounds on the vorticity
$$ \omega \coloneqq \nabla \times u$$
near $(t_n,x_n)$.  The importance of working with the vorticity comes from the fact that it obeys the \emph{vorticity equation}
\begin{equation}\label{vorticity}
 \partial_t \omega = \Delta \omega - (u \cdot \nabla) \omega + (\omega \cdot \nabla) u
\end{equation}
which can be viewed as a variable coefficient heat equation (in which the lower order coefficients $u, \nabla u$ depend on the velocity field) for which the non-local effects of the pressure $p$ do not explicitly appear.  Using a quantitative version of unique continuation for backwards parabolic equations (see Proposition \ref{carl-second} for a precise statement) that can be established using Carleman inequalities, one can then obtain exponentially small, but still non-trivial, lower bounds\footnote{One can think of this as applying (a quantitative version) of unique continuation ``in the contrapositive''.  Similarly for the invocation of backwards uniqueness below.  Actually in practice the Carleman inequalities also require an additional term such as $|\nabla \omega(t,x)|^2$ in the integrand, but we ignore this term for sake of discussion.} for enstrophy-type quantities such as
$$ \int_{I_n} \int_{R_n \leq |x-x_n| \leq R'_n} |\omega(t,x)|^2\ dx dt $$
for various cylindrical annuli $I_n \times \{ x: R_n \leq |x-x_n| \leq R'_n \}$ surrounding $(t_n,x_n)$, with $R'_n$ a large multiple of $R_n$.  Crucially, one can set $R_n$ to be as large as one pleases (although the lower bound exhibits Gaussian decay in $R_n$).  In order to apply the Carleman inequalities, it is important that the time interval $I$ lies within one of the ``epochs of regularity'' in which one has good $L^\infty$ estimates for $u, \nabla u, \omega, \nabla \omega$, but this can be accomplished without much difficulty (mainly thanks to the energy dissipation term in the energy inequality).

For many choices of scale $R_n$ (a bit larger than $|t_0-t_n|^{1/2}$), one can use an ``energy pigeonholing argument'' (as used for instance by Bourgain \cite{bourgain}) to make the energy (or more precisely, a certain component of the enstrophy) small in an annular region $\{ x: R_n \leq |x-x_n| \leq R'_n \}$ at some time $t'_n$ a little bit to the past of $t_n$; by modifying the somewhat delicate analysis of local enstrophies from \cite{tao-local} that again takes advantage of the ``bounded total speed'' property, one can then propagate this smallness forward in time (at the cost of shrinking the annular region $\{ R_n \leq |x-x_n| \leq R'_n\}$ slightly), and in particular back up to time $t_0$, and parabolic regularity theory can then be used to obtain good $L^\infty$ estimates for $u, \nabla u, \omega, \nabla \omega$ in these regions.  This allows us to again use Carleman inequalities.  Specifically, by using the Carleman inequalities used to prove the backwards uniqueness result in \cite{ess} (see Section \ref{carleman-sec} for precise statements), one can then propagate the lower bounds on $I_n \times \{ x: R_n \leq |x-x_n| \leq R'_n \}$ forward in time until one returns to the original time $t_0$ of interest, eventually obtaining a small but nontrivial lower bound for quantities such as
$$ \int_{R_n \leq |x-x_n| \leq R'_n} |\omega(t_0,x)|^2\ dx $$
(ignoring for this discussion some slight adjustments to the scales $R_n,R'_n$ that occur during this argument), which after some routine manipulations (and using the fact that $(t_n,x_n)$ lies in the parabolic domain of dependence of $(t_0,x_0)$) also gives a lower bound on quantities like
$$ \int_{R_n \leq |x-x_0| \leq R'_n} |u(t_0,x)|^3\ dx.$$
Crucially, this lower bound is uniform in $n$.  If one now lets $n$ vary, the annuli $\{ R_n \leq |x-x_0| \leq R'_n \}$ end up becoming disjoint for widely separated $n$, and one can eventually contradict \eqref{u3} at time $t=t_0$ if $N_0$ is large enough.  

\begin{figure} [t]
\centering
\includegraphics[width=5in]{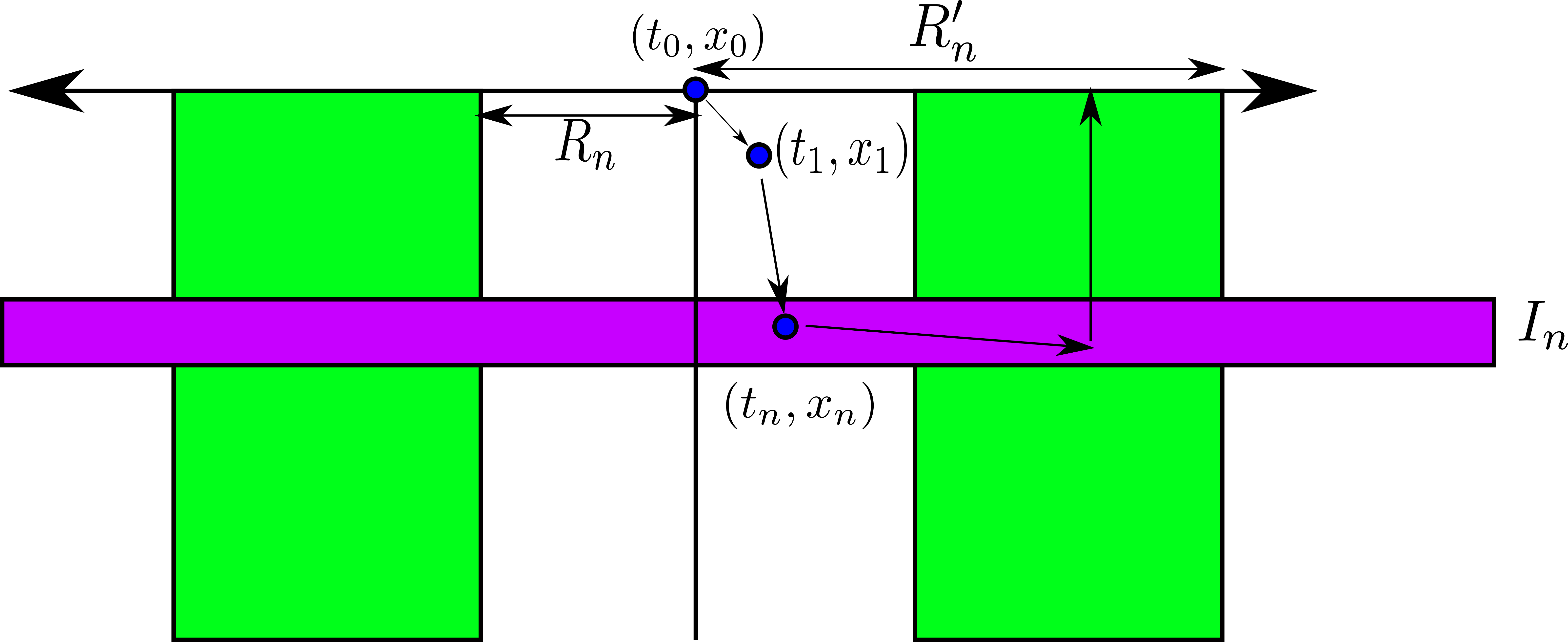}
\caption{A schematic depiction of the main argument.  Starting with a concentration of critical norm at a point $(t_0,x_0)$ in spacetime, one propagates this concentration backwards in time to generate concentrations at further points $(t_n,x_n)$ in spacetime.  Restricting attention to an epoch of regularity $I_n \times \R^3$ (depicted here in purple), Carleman estimates are then used to establish lower bounds on the vorticity at other locations in space, and in particular where the epoch intersects an ``annulus of regularity'' (depicted in green) arising from an energy (or enstrophy) pigeonholing argument.  A further application of Carleman estimates are then used to establish a lower bound on the vorticity (or velocity) in the annular region at time $t=t_0$, thus demonstrating a lack of compactness of the solution at this time which can be used to obtain a contradiction when $N_0$ (or more precisely the scale-invariant quantity $N_0^2T$, where $T$ is the lifespan of the solution) is large enough, by letting $n$ vary.}
\label{fig:scheme}
\end{figure}

\begin{remark}
The triply exponential nature of the bounds in Theorem \ref{main-1} (which is of course closely tied to the triply logarithmic improvement to Theorem \ref{opp} in Theorem \ref{main-2}) can be explained as follows.  One exponential factor comes from the Bourgain energy pigeonholing argument to locate a good spatial scale $R$.  A second exponential factor arises from the Carleman inequalities.  The third exponential arises from locating enough disjoint spatial scales $R_n$ to contradict \eqref{u3}.  It seems that substantially new ideas would be needed in order to improve significantly upon this triple exponential bound.
\end{remark}

\begin{remark}
Of course, by Sobolev embedding, the $L^3_x(\R^3)$ norm in the above theorems can be replaced by the critical homogeneous Sobolev norm $\dot H^{1/2}_x(\R^3)$.  It is likely that the arguments here can also be adapted to handle other critial Besov or Lorentz spaces (as long as the secondary exponent of such spaces is finite, so that the critical norm cannot simultaneously have a substantial presence at an unbounded number of scales), but we will not pursue this question here; based on Theorem \ref{main-2}, it is also reasonable to conjecture that the Orlicz norm $\| u(t) \|_{L^3 (\log\log\log L)^{-c}(\R^3)}$ of $u$ also must blow up as $t \to T_*^-$ for some absolute constant $c>0$.  On the other hand, our argument relies heavily in many places on the fact that we are working in three dimensions.  It may be possible to obtain a higher-dimensional analogue of our results by finding quantitative versions of the argument in \cite{dong}, but we do not pursue this question here.  Similarly, our arguments do not directly allow us to replace the limit superior in Theorem \ref{opp} with a limit, as is done in \cite{seregin} (see also \cite{albritton}); again, it may be possible to also find quantitative analogues of these results, but we do not pursue this matter here. 
\end{remark}


\section{Notation}\label{notation-sec}

We use the notation $X = O(Y)$, $X \lesssim Y$, or $Y \gtrsim X$ to denote the bound $|X| \leq CY$ for some absolute constant $C > 0$.  If we need the implied constant $C$ to depend on parameters we shall indicate this by subscripts, for instance $X \lesssim_j Y$ denotes the bound $|X| \leq C_j Y$ where $C_j$ depends only on $j$.  

Throughout this paper we will need a sufficiently large absolute constant $C_0$, which will remain fixed throughout the paper.  For instance $C_0 = 10^5$ would suffice throughout our paper, if one worked out all the implied constants in the exponents carefully.  

If $I \subset \R$ is a time interval, we use $|I|$ to denote its length.  If $x_0 \in \R^3$ and $R > 0$, we use $B(x_0,R)$ to denote the ball $\{ x \in \R^3: |x-x_0| \leq R \}$, and if $B = B(x_0,R)$ is such a ball, we use $kB = B(x_0,kR)$ to denote its dilates for any $k > 0$.

We use the mixed Lebesgue norms
$$ \| u \|_{L^q_t L^r_x(I \times \R^3)} \coloneqq \left( \int_I \| u(t) \|_{L^r_x(\R^3)}^q\ dt \right)^{1/q}$$
where
$$ \| u(t) \|_{L^r_x(\R^3)} \coloneqq  \left( \int_{\R^3} |u(t,x)|^r\ dx \right)^{1/r}$$
with the usual modifications when $q=\infty$ or $r=\infty$.  For any measurable subset $\Omega \subset I \times \R^3$, we write $\| u \|_{L^q_t L^r_x(\Omega)}$ for $\| u 1_\Omega \|_{L^q_t L^r_x(I \times \R^3)}$, where $1_\Omega$ is the indicator function of $\Omega$.

Given a Schwartz function $f: \R^3 \to \R$, we define the Fourier transform
$$ \hat f(\xi) \coloneqq \int_{\R^3} f(x) e^{-2\pi i \xi \cdot x}\ dx$$
and then for any $N>0$ we define the Littlewood-Paley projection $P_{\leq N}$ by the formula
$$ \widehat{P_{\leq N} f}(\xi) \coloneqq \varphi(\xi/N) \hat f(\xi)$$
where $\varphi: \R^3 \to \R$ is a fixed bump function supported on $B(0,1)$ that equals $1$ on $B(0,1/2)$.  We also define the companion Littlewood-Paley projections
\begin{align*}
P_N &\coloneqq P_N - P_{N/2} \\
P_{>N} &\coloneqq 1 - P_{\leq N} \\
\tilde P_N &\coloneqq P_{2N} - P_{N/4}
\end{align*}
where $1$ denotes the identity operator; thus for instance $P_{\leq N} f= \sum_{k=0}^\infty P_{2^{-k} N} f$ and $P_{>N} f = \sum_{k=1}^\infty P_{2^k N} f$ for Schwartz $f$ (with the convergence in a locally uniform sense).  Also we have $P_N = P_N \tilde P_N$.  
These operators can also be applied to vector-valued Schwartz functions by working component by component.  These operators commute with other Fourier multipliers such as the Laplacian $\Delta$ and its inverse $\Delta^{-1}$, partial derivatives $\partial_i$, heat propagators $e^{t\Delta}$, and the Leray projection $\mathbb{P} \coloneqq - \nabla \times \Delta^{-1} \nabla \times$ to divergence-free vector fields.  To estimate such multipliers, we use the following general estimate:

\begin{lemma}[Multiplier theorem]\label{mult}  Let $N>0$, and let $m: \R^3 \to \C$ be a smooth function supported on $B(0,N)$ that obeys the bounds
$$ |\nabla^j m(\xi)| \leq M N^{-j}$$
for all $0 \leq j \leq 100$ and some $M>0$.  Let $T_m$ denote the associated Fourier multiplier, thus
$$ \widehat{T_m f}(\xi) \coloneqq m(\xi) f(\xi).$$
Then one has
\begin{equation}\label{global}
 \| T_m f \|_{L^q(\R^3)} \lesssim M N^{\frac{3}{p}-\frac{3}{q}} \|f\|_{L^p(\R^3)}
\end{equation}
whenever $1 \leq p \leq q \leq \infty$ and $f: \R^3 \to \R$ is a Schwartz function.  More generally, if $\Omega \subset \R^3$ is an open subset of $\R^3$, $A \geq 1$, and $\Omega_{A/N} \coloneqq \{ x \in \R^3: \mathrm{dist}(x,\Omega) < A/N\}$ denotes the $A/N$-neighbourhood of $\Omega$, then we have a local version
\begin{equation}\label{local}
 \| T_m f \|_{L^{q_1}(\Omega)} \lesssim M N^{\frac{3}{p_1}-\frac{3}{q_1}} \|f\|_{L^{p_1}(\Omega_{A/N})} + A^{-50} M |\Omega|^{\frac{1}{q_1}-\frac{1}{q_2}} N^{\frac{3}{p_2}-\frac{3}{q_2}} \|f\|_{L^{p_2}(\R^3)}
\end{equation}
of the above estimate, whenever $1 \leq p_1 \leq q_1 \leq \infty$ and $1 \leq p_2 \leq q_2 \leq \infty$ are such that $q_2 \geq q_1$, and $|\Omega|$ denotes the volume of $\Omega$.
\end{lemma}

By the usual limiting arguments, one can replace the hypothesis that $f$ is Schwartz with the requirement that $f$ lie in $L^p$.  Also one can extend this theorem to vector-valued $f: \R^3 \to \R^3$ by working component by component.   In practice, the $A^{-50}$ factor will ensure that the second term on the right-hand side of \eqref{local} is negligible compared to the first, and can be ignored on a first reading.

\begin{proof}  By homogeneity we can normalise $M=1$; by scaling (or dimensional analysis) we may also normalise $N=1$.  We can write $T_m f$ as a convolution $T_m f = f * K$ of $f$ with the kernel
$$ K(x) \coloneqq \int_{\R^3} m(\xi) e^{2\pi i \xi \cdot x}\ d\xi.$$
By repeated integration by parts we obtain the bounds $K(x) \lesssim (1+|x|)^{-90}$ (say), so in particular $\|K\|_{L^r(\R^3)} \lesssim 1$ for all $1 \leq r \leq \infty$.  From Young's convolution inequality we then conclude that
$$ \| T_m f \|_{L^q(\R^3)} \lesssim \|f\|_{L^p(\R^3)}$$
giving \eqref{global}.  To prove \eqref{local}, we see that the claim already follows from \eqref{global} when $f$ is supported in $L^p(\Omega_A)$, so by the triangle inequality we may assume that $f$ is supported on $\R^3 \backslash \Omega_A$.  In this case we may replace the convolution kernel $K$ by its restriction to the complement of $B(0,A)$, which allows us to improve the bound on the $L^r$ norm of the kernel to (say) $O(A^{-50})$.  The claim follows from Young's convolution inequality, after first using H\"older's inequality to bound $ \| T_m f \|_{L^{q_1}(\Omega)} \leq |\Omega|^{\frac{1}{q_1}-\frac{1}{q_2}} \| T_m f \|_{L^{q_2}(\Omega)}$.
\end{proof}

Thus for instance, we have the Bernstein inequalities
\begin{equation}\label{bern}
 \| \nabla^j f \|_{L^q(\R^3)} \lesssim_j N^{j + \frac{3}{p} - \frac{3}{q}} \|f\|_{L^p(\R^3)}
\end{equation}
whenever $1 \leq p \leq q \leq \infty$, $j \geq 0$, and $f$ is a Schwartz function whose Fourier transform is supported on $B(0,N)$, as can be seen by writing $f = P_{\leq 2N} f$ and applying Lemma \ref{mult}.  In a similar spirit, one has
\begin{equation}\label{bern-2}
 \| P_N e^{t\Delta} \nabla^j f \|_{L^q(\R^3)} \lesssim_j \exp( - N^2 t / 20 ) N^{j + \frac{3}{p} - \frac{3}{q}} \|f\|_{L^p(\R^3)}
\end{equation}
for any $t>0$ and any Schwartz $f$.  Summing this, we obtain the standard heat kernel bounds
\begin{equation}\label{bern-3}
 \| e^{t\Delta} \nabla^j f \|_{L^q(\R^3)} \lesssim_j t^{-\frac{j}{2} - \frac{3}{2p} + \frac{3}{2q}} \|f\|_{L^p(\R^3)}.
\end{equation}

\section{Basic estimates}

The purpose of this section is to establish the following initial bounds for $L^\infty_t L^3_x$-bounded solutions to the Navier-Stokes equations.

\begin{proposition}[Initial estimates]\label{basic}  Let $u: [t_0-T,t_0] \times \R^3 \to \R^3$, $p: [t_0-T,t_0] \times \R^3 \to \R$ be a classical solution to Navier-Stokes that obeys the bound
\begin{equation}\label{able}
 \| u \|_{L^\infty_t L^3_x([t_0-T,t_0] \times \R^3)} \leq A.
\end{equation}
for some $A \geq C_0$.  We adopt the notation
$$ A_j \coloneqq A^{C_0^j}$$
for all $j$, thus $A_0=A$ and $A_{j+1}=A_j^{C_0}$.
\begin{itemize}
\item[(i)]  (Pointwise derivative estimates) For any $(t,x) \in [t_0-T/2, t_0] \times \R^3$ and $N > 0$, we have
\begin{equation}\label{pant}
P_N u(t,x) = O( A N ); \quad \nabla P_N u(t,x) = O( A N^2 ); \quad \partial_t P_N u(t,x) = O( A^2 N^3 );
\end{equation}
similarly, the vorticity $\omega \coloneqq \nabla \times u$ obeys the bounds
\begin{equation}\label{pant-2}
P_N \omega(t,x) = O( A N^2 ); \quad \nabla P_N \omega(t,x) = O( A N^3 ); \quad \partial_t P_N \omega(t,x) = O( A^2 N^4 ).
\end{equation}
\item[(ii)]  (Bounded total speed) For any interval $I$ in $[t_0-T/2, t_0]$, one has
\begin{equation}\label{bts}
 \| u \|_{L^1_t L^\infty_x( I \times \R^3 )} \lesssim A^{4} |I|^{1/2}.
\end{equation}
\item[(iii)]  (Epochs of regularity) For any interval $I$ in $[t_0-T/2, t_0]$, there is a subinterval $I' \subset I$ with $|I'| \gtrsim A^{-8} |I|$ such that
$$ \| \nabla^j u\|_{L^\infty_t L^\infty_x(I' \times \R^3)} \lesssim A^{O(1)} |I|^{-(j+1)/2}$$
and
$$ \| \nabla^j \omega\|_{L^\infty_t L^\infty_x(I'  \times \R^3)} \lesssim A^{O(1)} |I|^{-(j+2)/2}$$
for $j=0,1$.
\item[(iv)]  (Back propagation)  Let $(t_1,x_1) \in [t_0-T/2,t_0] \times \R^3$ and $N_1 \geq A_3 T^{-1/2}$ be such that
\begin{equation}\label{n1}
 |P_{N_1} u(t_1,x_1)| \geq A_1^{-1} N_1.
\end{equation}
Then there exists $(t_2,x_2) \in [t_0-T,t_1] \times \R^3$ and $N_2 \in [A_2^{-1} N_1, A_2 N_1]$ such that
$$ A_3^{-1} N_1^{-2} \leq t_1-t_2 \leq A_3 N_1^{-2}$$
and
$$ |x_2-x_1| \leq A_4 N_1^{-1}$$
and
\begin{equation}\label{n2}
 |P_{N_2} u(t_2,x_2)| \geq A_1^{-1} N_2.
\end{equation}
\item[(v)] (Iterated back propagation)  Let $x_0 \in \R^3$ and $N_0 > 0$ be such that
$$ |P_{N_0} u(t_0,x_0)| \geq A_1^{-1} N_0.$$
Then for every $A_4 N_0^{-2} \leq T_1 \leq A_4^{-1} T$, there exists 
$$ (t_1,x_1) \in [t_0-T_1, t_0 - A_3^{-1} T_1] \times \R^3$$
and
$$ N_1 = A_3^{O(1)} T_1^{-1/2}$$
such that
$$ x_1 = x_0 + O( A_4^{O(1)} T_1^{1/2})$$
and
$$ |P_{N_1} u(t_1,x_1)| \geq A_1^{-1} N_1.$$
\item[(vi)]  (Annuli of regularity)  If $0 < T' < T/2$, $x_0 \in \R^3$, and $R_0 \geq (T')^{1/2}$, then there exists a scale
$$ R_0 \leq R \leq \exp(A_6^{O(1)}) R_0$$
such that on the region
$$ \Omega := \{ (t,x) \in [t_0-T',t_0] \times \R^3: R \leq |x-x_0| \leq A_6 R \}$$
we have
$$ \| \nabla^j u \|_{L^\infty_t L^\infty_x(\Omega)} \lesssim A_6^{-2} (T')^{-(j+1)/2}$$
and
$$ \| \nabla^j \omega \|_{L^\infty_t L^\infty_x(\Omega)} \lesssim A_6^{-2} (T')^{-(j+2)/2}$$
for $j=0,1$.
\end{itemize}
\end{proposition}

As $C_0$ is assumed large, any polynomial combination of $A=A_0, A_1,\dots,A_{j-1}$ will be dominated by $A_j$ for any $j \geq 1$; we take advantage of this fact without comment in the sequel to simplify the estimates.  The various numerical powers of $A$ (or $A_j$) that appear in the above proposition are not of much significance, except that it is important for iterative purposes that the negative power $A_1^{-1}$ appearing in \eqref{n1} is exactly the same as the one appearing in \eqref{n2}. 

In the remainder of this section $t_0, T, A, u, p$ are as in Proposition \ref{basic}.  Our objective is now to establish the claims (i)-(vi).

We begin with the proof of (i).  It suffices to establish \eqref{pant}, as \eqref{pant-2} then follows from the Bernstein inequalities \eqref{bern}.  The first two claims of \eqref{pant} are immediate from \eqref{able} and \eqref{bern}.  For the final claim, we first apply the Leray projection $\mathbb{P}$ to \eqref{ns} to obtain the familiar equation
\begin{equation}\label{leray}
 \partial_t u = \Delta u - \mathbb{P} \nabla \cdot (u \otimes u)
\end{equation}
where the divergence $\nabla \cdot (u \otimes u)$ of the symmetric tensor $u \otimes u$ is expressed in coordinates as
$$ (\nabla \cdot (u \otimes u))_i = \partial_j (u_i u_j)$$
with the usual summation conventions.  We apply $P_N$ to both sides of \eqref{leray}.  From \eqref{able} and \eqref{bern} we have
$$ \| P_N \Delta u(t) \|_{L^\infty_x(\R^3)} \lesssim N^3 A.$$
From \eqref{able} and H\"older we have $\| u \otimes u(t) \|_{L^{3/2}_x(\R^3)} \lesssim A^2$, hence by Lemma \ref{mult} we have
$$ \| P_N \mathbb{P} \nabla \cdot (u \otimes u)(t) \|_{L^\infty_x(\R^3)} \lesssim N^3 A^2,$$
and the final claim of \eqref{pant} follows from the triangle inequality.

Now we prove (ii), (iii).  It is not difficult to see that these estimates are invariant with respect to time translation (shifting $I, t_0, u$ accordingly) and also rescaling (adjusting $T, t_0, I, u$ accordingly).  Hence we may assume without loss of generality that $I = [0,1] \subset [t_0-T/2,t_0]$, which implies that $[-1,1] \subset [t_0-T,t_0]$.  

It will be convenient to remove\footnote{See also \cite{cal} for a similar technique to apply energy methods to Navier-Stokes solutions that lie in a function space other than $L^2_x$.} a linear component from $u$, as it is not well controlled in $L^2_x$ type spaces.  Namely, on $[-1,1] \times \R^3$ we split $u = u^{\lin} + u^{\nl}$, where $u^{\lin}$ is the linear solution 
\begin{equation}\label{linear}
 u^{\lin}(t) \coloneqq e^{(t+1)\Delta} u(-1)
\end{equation}
and $u^{\nl} \coloneqq u - u^{\lin}$ is the nonlinear component.  From \eqref{able} we have
\begin{equation}\label{able-2}
 \| u^{\lin} \|_{L^\infty_t L^3_x([-1,1] \times \R^3)}, \| u^{\nl} \|_{L^\infty_t L^3_x([-1,1] \times \R^3)} \lesssim A.
\end{equation}
From \eqref{leray} and Duhamel's formula one has
$$ u^{\nl}(t) = - \int_{-1}^t e^{(t-t')\Delta} \mathbb{P} \nabla \cdot (u \otimes u)(t')\ dt'.$$
From \eqref{able}, $u \otimes u$ has an $L^{3/2}_x(\R^3)$ norm of $O(A^2)$.  From \eqref{bern-3}, the operator $e^{(t-t')\Delta} \mathbb{P} \nabla \cdot$ maps $L^{3/2}_x$ to $L^2_x$ with an operator norm of $(t-t')^{-3/4}$.  From Minkowski's inequality we conclude an energy bound for the nonlinear component:
\begin{equation}\label{vl2}
 \| u^{\nl} \|_{L^\infty_t L^2_x([-1,1] \times \R^3)} \lesssim A^2.
\end{equation}
We now restrict attention to the slab $[-1/2,1] \times \R^3$.  Here $t+1$ lies between $1/2$ and $2$, and we can use \eqref{able}, \eqref{linear}, and \eqref{bern-3} to obtain very good bounds on $u^{\lin}$ (but only in spaces with an integrability exponent greater than or equal to $3$).  More precisely, we have
\begin{equation}\label{very}
 \| \nabla^j u^{\lin} \|_{L^\infty_t L^p_x([-1/2,1] \times \R^3)} \lesssim_j A
\end{equation}
for any $3 \leq p \leq \infty$ and $j \geq 0$.

To exploit the bound \eqref{vl2}, we use the energy method.  Since $u^{\lin}$ solves the heat equation $\partial_t u^{\lin} = \Delta u^{\lin}$, we can subtract this from \eqref{ns} to conclude that
\begin{equation}\label{vi}
 \partial_t u^{\nl} = \Delta u^{\nl} - \nabla \cdot (u \otimes u) - \nabla p.
\end{equation}
Taking inner products with $u^{\nl}$, which is divergence-free, and integrating by parts, we conclude that
$$ \frac{1}{2} \partial_t \int_{\R^3} |u^{\nl}|^2\ dx = - \int_{\R^3} |\nabla u^{\nl}|^2\ dx + \int_{\R^3} (\nabla u^{\nl}) \cdot (u \otimes u)\ dx$$
where the quantity $(\nabla u^{\nl}) \cdot (u \otimes u)$ is defined in coordinates as
$$ (\nabla u^{\nl}) \cdot (u \otimes u) \coloneqq (\partial_i u^{\nl}_j) u_i u_j.$$
From the divergence-free nature of $u^{\nl}$ and integration by parts we have
$$ \int_{\R^3} (\nabla u^{\nl}) \cdot (u^{\nl} \otimes u^{\nl})\ dx = 0$$
and hence
$$ \frac{1}{2} \partial_t \int_{\R^3} |u^{\nl}|^2\ dx = - \int_{\R^3} |\nabla u^{\nl}|^2\ dx + \int_{\R^3} (\nabla u^{\nl}) \cdot (u \otimes u - u^{\nl}\otimes u^{\nl})\ dx.$$
Integrating this on $[-1/2,1]$ using \eqref{vl2} we conclude that
$$ \int_{-1/2}^1 \int_{\R^3} |\nabla u^{\nl}|^2\ dx dt \lesssim A^2 + \int_{-1/2}^1 \int_{\R^3} |\nabla u^{\nl}|  |u \otimes u - u^{\nl}\otimes u^{\nl}|\ dx dt,$$
and hence by Young's inequality
$$ \int_{-1/2}^1 \int_{\R^3} |\nabla u^{\nl}|^2\ dx dt \lesssim A^2 + \int_{-1/2}^1 \int_{\R^3} |u \otimes u - u^{\nl}\otimes u^{\nl}|^2\ dx dt.$$
Splitting $u \otimes u - u^{\nl} \otimes u^{\nl} = u_l \otimes u + u^{\nl} \otimes u_l$ and using \eqref{able}, \eqref{able-2}, \eqref{very} (with $p=6$, $j=0$) and H\"older's inequality, one has
$$ \int_{-1/2}^1 \int_{\R^3} |u \otimes u - u^{\nl}\otimes u^{\nl}|^2\ dx dt \lesssim A^4 $$
and thus
\begin{equation}\label{122}
 \int_{-1/2}^1 \int_{\R^3} |\nabla u^{\nl}|^2\ dx dt \lesssim A^4.
\end{equation}
By Plancherel's theorem this implies in particular that
\begin{equation}\label{equiv}
 \sum_N N^2 \| P_N u^{\nl} \|_{L^2_t L^2_x([-1/2,1] \times \R^3)}^2 \lesssim A^4
\end{equation}
where $N$ ranges over powers of two.  Also, from Sobolev embedding one has
\begin{equation}\label{v26}
 \| u^{\nl} \|_{L^2_t L^6_x([-1/2,1] \times \R^3)} \lesssim A^2.
\end{equation}
We are now ready to establish the bounded total speed property (ii), which is a variant of \cite[Proposition 9.1]{tao-local}.  If $t \in [0,1]$ and $N \geq 1$ is a power of two, we see from \eqref{leray} and Duhamel's formula that
$$ P_N u^{\nl}(t) = e^{(t+\frac{1}{2})\Delta} P_N u^{\nl}\left(-\frac{1}{2}\right) - \int_{-1/2}^t P_N e^{(t-t')\Delta} \mathbb{P} \nabla \cdot \tilde P_N (u \otimes u)(t')\ dt'.$$
From \eqref{bern-2} the operator $P_N e^{(t-t')\Delta} \mathbb{P} \nabla \cdot$ has an operator norm of $O( N \exp( - N^2 (t-t') / 20 ) )$ on $L^\infty_x$, while from \eqref{able-2}, \eqref{bern-2} we see that $e^{(t+\frac{1}{2})\Delta} P_N v(-\frac{1}{2})$ has an $L^\infty_x$ norm of $O( A N \exp( - N^2 / 20 ) )$.  Thus by Young's inequality
$$ \| P_N u^{\nl} \|_{L^1_t L^\infty_x([0,1] \times \R^3)} \lesssim A N \exp(-N^2/20) + N^{-1} \| \tilde P_N (u \otimes u) \|_{L^1_t L^\infty_x([-1/2,1] \times \R^3)}.$$
We split $u \otimes u = u^{\lin} \otimes u^{\lin} + u^{\lin} \otimes u^{\nl} + u^{\nl} \otimes u^{\lin} + u^{\nl} \otimes u^{\nl}$.  From \eqref{very} one has
$$ \| \tilde P_N (u^{\lin} \otimes u^{\lin}) \|_{L^1_t L^\infty_x([-1/2,1] \times \R^3)} \lesssim A^2.$$
From \eqref{bern}, H\"older's inequality, and \eqref{very}, \eqref{v26} one has
\begin{align*}
 \| \tilde P_N (u^{\lin} \otimes u^{\nl}) \|_{L^1_t L^\infty_x([-1/2,1] \times \R^3)} &\lesssim 
N^{1/2}
 \| u^{\lin} \otimes u^{\nl} \|_{L^1_t L^6_x([-1/2,1] \times \R^3)} \\
&\lesssim A^3 N^{1/2}.
\end{align*}
Similarly with $u^{\lin} \otimes u^{\nl}$ replaced by $u^{\nl} \otimes u^{\lin}$.  We then split $u^{\nl} \otimes u^{\nl} = P_{\leq N} u^{\nl} \otimes P_{\leq N} u^{\nl} + P_{\leq N} u^{\nl} \otimes P_{>N} u^{\nl} + P_{> N} u^{\nl} \otimes P_{\leq N} u^{\nl} + P_{>N} u^{\nl} \otimes P_{>N} u^{\nl}$. We have from H\"older that
\begin{align*}
\| P_{\leq N} u^{\nl} \otimes P_{\leq N} u^{\nl} \|_{L^1_t L^\infty_x([-1/2,1] \times \R^3)} &\lesssim
\| P_{\leq N} u^{\nl} \|_{L^2_t L^\infty_x([-1/2,1] \times \R^3)}^2 \\
\| P_{\leq N} u^{\nl} \otimes P_{>N} u^{\nl}\|_{L^1_t L^2_x([-1/2,1] \times \R^3)}, \quad & \\
\| P_{> N} u^{\nl} \otimes P_{\leq N} u^{\nl} \|_{L^1_t L^2_x([-1/2,1] \times \R^3)} &\lesssim
\| P_{\leq N} u^{\nl} \|_{L^2_t L^\infty_x([-1/2,1] \times \R^3)} \\
&\quad \times\| P_{> N} u^{\nl} \|_{L^2_t L^2_x([-1/2,1] \times \R^3)} \\
\| P_{> N} u^{\nl} \otimes P_{>N} u^{\nl} \|_{L^1_t L^1_x([-1/2,1] \times \R^3)}&\lesssim
\| P_{> N} u^{\nl} \|_{L^2_t L^2_x([-1/2,1] \times \R^3)}^2
\end{align*}
and hence by \eqref{bern}, the triangle inequality, and Young's inequality
$$ \| \tilde P_N (u \otimes u) \|_{L^1_t L^\infty_x([-1/2,1] \times \R^3)} \lesssim \| P_{\leq N} u^{\nl} \|_{L^2_t L^\infty_x([-1/2,1] \times \R^3)}^2 + N^3
\| P_{> N} u^{\nl} \|_{L^2_t L^2_x([-1/2,1] \times \R^3)}^2.$$
Putting all this together, we conclude that
\begin{align*}
 \| P_N u^{\nl} \|_{L^1_t L^\infty_x([0,1] \times \R^3)} &\lesssim A^3 N^{-1/2}\\
&\quad + N^{-1} \| P_{\leq N} u^{\nl} \|_{L^2_t L^\infty_x([-1/2,1] \times \R^3)}^2 \\
&\quad + N^2 \| P_{> N} u^{\nl} \|_{L^2_t L^2_x([-1/2,1] \times \R^3)}^2.
\end{align*}
By \eqref{bern} and Cauchy-Schwarz we have
\begin{align*}
 \| P_{\leq N} u^{\nl} \|_{L^2_t L^\infty_x([-1/2,1] \times \R^3)}^2 &\lesssim \left(\sum_{N' \leq N} (N')^{3/2} \| P_{N'} u^{\nl} \|_{L^2_t L^2_x([-1/2,1] \times \R^3)}\right)^2\\
&\lesssim N^{1/2} \sum_{N' \leq N} (N')^{2} \| P_{N'} u^{\nl} \|_{L^2_t L^2_x([-1/2,1] \times \R^3)}^2
\end{align*}
where $N'$ ranges over powers of two, while from Plancherel's theorem one has
$$ \| P_{> N} u^{\nl} \|_{L^2_t L^2_x([-1/2,1] \times \R^3)}^2 \lesssim \sum_{N' > N} \| P_{N'} u^{\nl} \|_{L^2_t L^2_x([-1/2,1] \times \R^3)}^2.$$
Summing in $N$, and using the triangle inequality followed by \eqref{equiv}, we conclude that
$$ \| P_{\geq 1} u^{\nl} \|_{L^1_t L^\infty_x([0,1] \times \R^3)} \lesssim A^3 + \sum_{N'} (N')^2 \| P_{N'} u^{\nl} \|_{L^2_t L^2_x([-1/2,1] \times \R^3)}^2 \lesssim A^4.$$
From \eqref{able-2} and \eqref{bern} we also have
$$ \| u^{\lin} \|_{L^1_t L^\infty_x([0,1] \times \R^3)}, \| P_{< 1} u^{\nl} \|_{L^1_t L^\infty_x([0,1] \times \R^3)} \lesssim A$$
and we conclude
$$ \| u \|_{L^1_t L^\infty_x([0,1] \times \R^3)} \lesssim A^4$$
which gives (ii).

Now we establish (iii).  For $t \in [0,1]$ we define the enstrophy-type quantity
$$ E(t) \coloneqq \frac{1}{2} \int_{\R^3} |\nabla u^{\nl}(t,x)|^2\ dx,$$
Taking the gradient of \eqref{vi} and then taking the inner product with $\nabla u^{\nl}$, we see upon integration by parts that
$$ \partial_t E(t) = - \int_{\R^3} |\nabla^2 u^{\nl}|^2\ dx + \int_{\R^3} \Delta u^{\nl} \cdot (\nabla \cdot (u \otimes u))\ dx$$
and hence by Young's inequality
$$\partial_t E(t) \leq - \frac{1}{2} \| \nabla^2 u^{\nl} \|_{L^2_x(\R^3)}^2 + O\left( \| \nabla \cdot (u \otimes u) \|_{L^2_x(\R^3)}^2 \right).$$
By the Leibniz rule and H\"older's inequality, one has
$$ \| \nabla \cdot (u \otimes u) \|_{L^2_x(\R^3)} \lesssim \| u \|_{L^6_x(\R^3)} \| \nabla u \|_{L^3_x(\R^3)}.$$
From \eqref{very} and the triangle inequality one has
$$\| u \|_{L^6_x(\R^3)}  \lesssim A + \| u^{\nl} \|_{L^6_x(\R^3)}$$
and
$$\| \nabla u \|_{L^3_x(\R^3)}  \lesssim A + \| \nabla u^{\nl} \|_{L^3_x(\R^3)}$$
while from Sobolev embedding and H\"older one has
$$ \| u^{\nl} \|_{L^6_x(\R^3)} \lesssim \| \nabla u^{\nl} \|_{L^2_x(\R^3)} \lesssim E(t)^{1/2}$$
and
$$ \| \nabla u^{\nl} \|_{L^3_x(\R^3)} \lesssim \| \nabla u^{\nl} \|_{L^2_x(\R^3)}^{1/2} \| \nabla^2 u^{\nl} \|_{L^2_x(\R^3)}^{1/2}
\lesssim E(t)^{1/4} \| \nabla^2 u^{\nl} \|_{L^2_x(\R^3)}^{1/2}.$$
We conclude that
$$\partial_t E(t) \leq -\frac{1}{2} \| \nabla^2 u^{\nl} \|_{L^2_x(\R^3)}^2 + O\left( (A^2 + E(t)) (A^2 + E(t)^{1/2} \|\nabla^2 u^{\nl}\|_{L^2_x(\R^3)} ) \right)$$
and hence by Young's inequality
\begin{equation}\label{pet}
\partial_t E(t) \leq -\frac{1}{4} \| \nabla^2 u^{\nl} \|_{L^2_x(\R^3)}^2 + O( (A^2 + E(t)) A^2 + (A^2 + E(t))^2 E(t) ).
\end{equation}
In particular we have
\begin{equation}\label{eat} \partial_t E(t) \leq O( A^4 + A^4 E(t) + E(t)^3 ).
\end{equation}
From \eqref{122} we have
$$ \int_0^1 E(t)\ dt \lesssim A^4,$$
and hence by the pigeonhole principle, we can find a time $t_1 \in [0,1/2]$ such that
$$ E(t_1)  \lesssim A^4.
$$
A standard continuity argument using \eqref{eat} then gives $E(t) \lesssim A^4$ for $t \in [t_1, t_1 + c A^{-8}] = [\tau(0), \tau(1)]$, where $\tau(s) \coloneqq t_1 + scA^{-8}$ and $c>0$ is a small absolute constant.  Inserting this back into \eqref{pet} one has
$$ \partial_t E(t) \leq -\frac{1}{4} \| \nabla^2 u^{\nl} \|_{L^2_x(\R^3)}^2 + O( A^{12} ) $$
and hence by the fundamental theorem of calculus
\begin{equation}\label{nog}
 \int_{\tau(0)}^{\tau(1)} \int_{\R^3} |\nabla^2 u^{\nl}|^2\ dx dt \lesssim A^4.
\end{equation}
Thus we have
\begin{equation}\label{man}
\| \nabla u^{\nl} \|_{L^\infty_t L^2_x([\tau(0), \tau(1)] \times \R^3)} + \| \nabla^2 u^{\nl} \|_{L^2_t L^2_x([\tau(0), \tau(1)] \times \R^3)} \lesssim A^2.
\end{equation}
From the Gagliardo-Nirenberg inequality
\begin{equation}\label{unl}
 \| u^{\nl} \|_{L^\infty_x} \lesssim \| \nabla u^{\nl} \|_{L^2_x}^{1/2} \| \nabla^2 u^{\nl} \|_{L^2_x}^{1/2}
\end{equation}
and H\"older's inequality, one concludes in particular that
$$ \| u^{\nl} \|_{L^4_t L^\infty_x([\tau(0), \tau(1)] \times \R^3)} \lesssim A^2$$
and hence by \eqref{very}
\begin{equation}\label{fract}
 \| u \|_{L^4_t L^\infty_x([\tau(0), \tau(1)] \times \R^3)} \lesssim A^2;
\end{equation}
also from Sobolev embedding and \eqref{man} one has
$$ \| \nabla u^{\nl} \|_{L^2_t L^6_x([\tau(0), \tau(1)] \times \R^3)} \lesssim A^2$$
and hence by \eqref{very}
\begin{equation}\label{nob}
 \| \nabla u \|_{L^2_t L^6_x([\tau(0), \tau(1)] \times \R^3)} \lesssim A^2.
\end{equation}
These are subcritical regularity estimates and can now be iterated to obtain even higher regularity.  For $t \in [\tau(0.1), \tau(1)]$, we see from \eqref{leray} that
\begin{equation}\label{ler}
 u(t) = e^{(t-\tau(0)) \Delta} u(\tau(0)) - \int_{\tau(0)}^t e^{(t-t')\Delta} \mathbb{P} \nabla \cdot (u \otimes u)(t')\ dt'.
\end{equation}
From \eqref{bern-3} the operator $e^{(t-t')\Delta} \mathbb{P} \nabla  \cdot$ has norm $O( (t-t')^{-1/2} )$ on $L^\infty_x$, while $e^{(t-\tau(0))\Delta}$ maps $L^3_x$ to $L^\infty_x$ with norm $O( (t-\tau(0))^{-1/2} ) = O( A^{O(1)} )$.  We conclude from \eqref{able} that
$$ \|u(t) \|_{L^\infty_x(\R^3)} \lesssim A^{O(1)} + \int_{\tau(0)}^t (t-t')^{-1/2} \| u(t) \|_{L^\infty_x(\R^3)}^2\ dt'.$$
From \eqref{fract} and Young's convolution inequality, we conclude that
$$ \| u \|_{L^8_t L^\infty_x([\tau(0.1), \tau(1)] \times \R^3)} \lesssim A^{O(1)}$$
Repeating the above argument, we now also see for $t \in [\tau(0.2), \tau(1)]$ that
$$ \|u(t) \|_{L^\infty_x(\R^3)} \lesssim A^{O(1)} + \int_{\tau(0.1)}^t (t-t')^{-1/2} \| u(t) \|_{L^\infty_x(\R^3)}^2\ dt'$$
so from H\"older's inequality we conclude that
\begin{equation}\label{leo}
 \| u \|_{L^\infty_t L^\infty_x([\tau(0.2), \tau(1)] \times \R^3)} \lesssim A^{O(1)}.
\end{equation}
Now we differentiate \eqref{leray} to conclude that
$$
\nabla u(t) = \nabla e^{(t-\tau(0.2)) \Delta} u(\tau(0.2)) - \int_{\tau(0.2)}^t \nabla e^{(t-t')\Delta} \mathbb{P} \nabla \cdot (u \otimes u)(t')\ dt'$$
for $t \in [\tau(0.3), \tau(1)]$.  From \eqref{able}, the first term $\nabla e^{(t-\tau(0.2)) \Delta} u(\tau(0.2))$ has an $L^\infty_x$ norm of $O( A^{O(1)} )$.  From \eqref{bern-3}, the operator $\nabla e^{(t-t')\Delta} \mathbb{P}$ maps $L^6_x$ to $L^\infty_x$ with norm $O((t-t')^{-3/4})$, thus
$$ \| \nabla u(t) \|_{L^\infty_x(\R^3)} \lesssim A^{O(1)} + \int_{\tau(0.2)}^t (t-t')^{-3/4} \| \nabla \cdot (u \otimes u)(t') \|_{L^6_x(\R^3)}\ dt'.$$
From \eqref{nob}, \eqref{leo}, Leibniz and H\"older one has
$$\| \nabla \cdot (u \otimes u) \|_{L^2_t L^6_x([\tau(0.2), \tau(1)] \times \R^3)} \lesssim A^{O(1)}$$
and hence by fractional integration
$$ \| \nabla u \|_{L^4_t L^\infty_x([\tau(0.3), \tau(1)] \times \R^3)} \lesssim A^{O(1)}.$$
From this, \eqref{leo}, Leibniz, and H\"older one has
$$ \| \nabla \cdot (u \otimes u) \|_{L^4_t L^\infty_x([\tau(0.3), \tau(1)] \times \R^3)} \lesssim A^{O(1)}.$$
By \eqref{bern-3}, $\nabla e^{(t-t')\Delta} \mathbb{P}$ has an operator norm of $O((t-t')^{-1/2})$ on $L^\infty_x$, thus
$$ \| \nabla u(t) \|_{L^\infty_x(\R^3)} \lesssim A^{O(1)} + \int_{\tau(0.3)}^t (t-t')^{-1/2} \| \nabla \cdot (u \otimes u)(t') \|_{L^\infty_x(\R^3)}\ dt'$$
for $t \in [\tau(0.4), \tau(1)]$, and hence by H\"older's inequality
$$ \| \nabla u \|_{L^\infty_t L^\infty_x([\tau(0.4), \tau(1)] \times \R^3)} \lesssim A^{O(1)}.$$
From the vorticity equation \eqref{vorticity}, we now have
$$ \partial_t \omega = \Delta \omega + O( A^{O(1)} (|\omega| + |\nabla \omega|) )$$
on $[\tau(0.4), \tau(1)] \times \R^3$, and also $\omega = O(A^{O(1)})$ on this slab.  Standard parabolic regularity estimates (see e.g., \cite{lady}) then give
$$ \| \nabla \omega \|_{L^\infty_t L^\infty_x([\tau(0.5), \tau(1)] \times \R^3)} \lesssim A^{O(1)}.$$
Setting $I' \coloneqq [\tau(0.5), \tau(1)]$, we obtain the claim (iii).  We remark that it is also possible to control higher derivatives $\nabla^j u, \nabla^j \omega$ with $j>1$, for instance by using parabolic Schauder estimates in H\"older spaces, but we will not need to do so here.

Now we establish (iv).  Let $t_1,x_1,N_1$ be as in that part of the proposition.  By rescaling we may normalise $N_1=1$, and by translation invariance we may normalise $(t_1,x_1) = (0,0)$, so that $t_0-T \leq - \frac{T}{2} \leq -\frac{A_3^{2}}{2}$, so in particular $[-2A_3,0] \subset [t_0-T,t_0]$.   From \eqref{n1} we have
\begin{equation}\label{n1-norm}
 |P_1 u(0,0)| \geq A_1^{-1}.
\end{equation}
Assume for contradiction that the claim fails, then we have
$$ \| P_N u \|_{L^\infty_t L^\infty_x( [-A_3, -A_3^{-1}] \times B(0, A_4) )} \leq A_1^{-1} N$$
for all $A_2^{-1} \leq N \leq A_2$.  From \eqref{pant} and the fundamental theorem of calculus in time, we can enlarge the time interval to reach $t=0$, so that
$$ \| P_N u \|_{L^\infty_t L^\infty_x( [-A_3, 0] \times B(0, A_4) )} \lesssim A_1^{-1} N.$$

Suppose now that $N \geq A_2^{-1}$.  For $t \in [-A_3,0]$, we can use Duhamel's formula, \eqref{leray}, and the triangle inequality to write
\begin{align*}
\|P_N u(t)\|_{L^{3/2}_x(B(0, A_4))} &\leq \| e^{(t+2A_3)\Delta} P_N u(-2A_3) \|_{L^{3/2}_x(B(0,A_4))} \\
&\quad + \int_{-2A_3}^t \| e^{(t-t')\Delta} P_N \nabla \cdot (u(t') \otimes u(t')) \|_{L^{3/2}_x(\R^3)}\ dt'.
\end{align*}
From \eqref{bern-2}, $e^{(t+2A_3)\Delta} P_N$ has an operator norm of $O( \exp( - N^2 A_3 / 20 ) )$ on $L^3_x$, and 
$e^{(t-t')\Delta} P_N \nabla \cdot $ similarly has an operator norm of $O( N \exp( - N^2 (t-t')/20))$ on $L^{3/2}_x$.  Applying \eqref{able} and H\"older's inequality, we conclude that
$$ \|P_N u(t)\|_{L^{3/2}_x(B(0, A_4))} \lesssim A A_4 \exp( - N^2 A_3 / 20 ) + A^2 N^{-1}$$
and hence in the range $N \geq A_2^{-1}$ we have
\begin{equation}\label{meme}
 \| P_N u \|_{L^\infty_t L^{3/2}_x( [-A_3, 0] \times B(0, A_4) )} \lesssim A^2 N^{-1}.
\end{equation}

Now suppose that $N \geq A_2^{-1/2}$.  For $t \in [-A_3/2,0]$, we again use Duhamel's formula, \eqref{leray} and the triangle inequality to write
\begin{align*}
\|P_N u(t)\|_{L^1_x(B(0, A_4/2))} &\leq \| e^{(t+A_3)\Delta} P_N u(-A_3) \|_{L^1_x(B(0,A_4/2))} \\
&\quad + \int_{-A_3}^t \| e^{(t-t')\Delta} P_N \nabla \cdot \tilde P_N (u(t') \otimes u(t')) \|_{L^1_x(B(0,A_4/2))}\ dt'.
\end{align*}
From \eqref{bern-2}, \eqref{able}, and H\"older as before we have
$$ \| e^{(t+A_3)\Delta} P_N u(-A_3) \|_{L^1_x(B(0,A_4/2))} \lesssim A A_4^2 \exp( - N^2 A_3 / 40 ).$$
From \eqref{local} one has
\begin{align*} &\| e^{(t-t')\Delta} P_N \nabla \cdot \tilde P_N (u(t') \otimes u(t')) \|_{L^1_x(B(0,A_4/2))}
\lesssim N \exp(- N^2 (t-t')/20) \\
&\quad \times \left( \| \tilde P_N (u(t') \otimes u(t')) \|_{L^1_x(B(0,3A_4/4)} + A_4^{-50} A_4^{1/2} \| \tilde P_N (u(t') \otimes u(t')) \|_{L^{3/2}_x(\R^3)} \right)
\end{align*}
and hence by \eqref{able}
$$ \| P_N u \|_{L^\infty_t L^{1}_x( [-A_3/2, 0] \times B(0, A_4/2) )} \lesssim A_4^{-40} +
N^{-1} \| \tilde P_N (u(t') \otimes u(t')) \|_{L^\infty_t L^1_x([-A_3, 0]  \times B(0,3A_4/4))}.$$
Since $\tilde P_N( P_{\leq N/100} u(t') \otimes P_{\leq N/100} u(t') )$ vanishes, we can write
\begin{equation}\label{pind}
\tilde P_N (u(t') \otimes u(t')) = \tilde P_N (P_{>N/100} u(t') \otimes u(t')) + \tilde P_N (P_{\leq N/100} u(t') \otimes P_{>N/100} u(t')).
\end{equation}
From \eqref{local}, \eqref{able} we have
$$\| \tilde P_N (P_{>N/100} u(t') \otimes u(t')) \|_{L^\infty_t L^1_x([-A_3, 0]  \times B(0,3A_4/4))}
\lesssim \| P_{>N/100} u(t') \otimes u(t') \|_{L^\infty_t L^1_x([-A_3, 0]  \times B(0,A_4))} + A_4^{-40}.$$
From \eqref{meme} (and the triangle inequality) as well as \eqref{able} and H\"older's inequality, we thus have
$$\| \tilde P_N (P_{>N/100} u(t') \otimes u(t')) \|_{L^\infty_t L^1_x([-A_3, 0]  \times B(0,3A_4/4))} \lesssim A^3 N^{-1}.$$
Similarly for the other component of \eqref{pind}.  We conclude that
\begin{equation}\label{mold}
 \| P_N u \|_{L^\infty_t L^{1}_x( [-A_3/2, 0] \times B(0, A_4/2) )} \lesssim A^3 N^{-2}
\end{equation}
for all $N \geq A_2^{-1/2}$.

Now suppose that $A_2^{-1/3} \leq N \leq A_2^{1/3}$.  For $t \in [-A_3/3,0]$, we again use Duhamel's formula, \eqref{leray}, and the triangle inequality as before to write
\begin{align*}
& \|P_N u(t)\|_{L^2_x(B(0, A_4/4))} \leq \| e^{(t+A_3/2)\Delta} P_N u(-A_3/2) \|_{L^2_x(B(0,A_4/4))} \\
&\quad + \int_{-A_3/2}^t \| e^{(t-t')\Delta} P_N \nabla \cdot \tilde P_N (u(t') \otimes u(t')) \|_{L^2_x(B(0,A_4/4))}\ dt'.
\end{align*}
Arguing as before we have
$$ \| e^{(t+A_3/2)\Delta} P_N u(-A_3/2) \|_{L^2_x(B(0,A_4/4))} \lesssim A A_4^{1/2} \exp( - N^2 A_3 / 120 )$$
and
\begin{align*} &\| e^{(t-t')\Delta} P_N \nabla \cdot \tilde P_N (u(t') \otimes u(t')) \|_{L^2_x(B(0,A_4/4))} \lesssim
N^{5/2} \exp( -N^2 (t-t')/20) \\
&\quad \left( \| \tilde P_N (u(t') \otimes u(t')) \|_{L^1_x(B(0,A_4/3))} + A_4^{-50} N^{-1} \| \tilde P_N (u(t') \otimes u(t')) \|_{L^{3/2}_x(\R^3)} \right) 
\end{align*}
and thus
\begin{equation}\label{panu}
 \| P_N u \|_{L^\infty_t L^2_x( [-A_3/4, 0] \times B(0, A_4/4) )} \lesssim A_4^{-40} +
N^{1/2} \| \tilde P_N (u(t') \otimes u(t')) \|_{L^\infty_t L^1_x([-A_3/2, 0]  \times B(0,A_4/3))}.
\end{equation}
We can split $\tilde P_N (u(t') \otimes u(t'))$ into $O(1)$ paraproduct terms of the form $\tilde P_N(P_{N'} u(t') \otimes P_{\leq N/100} u(t'))$ where $N' \sim N$, $O(1)$ terms of the form $\tilde P_N(P_{\leq N/100} u(t') \otimes P_{N'} u(t') )$, and a sum of the form $\sum_{N_1 \sim N_2 \gtrsim N} 
\tilde P_N(P_{N_1} u(t') \otimes P_{N_2} u(t'))$.  For the ``high-low'' term $\tilde P_N(P_{N'} u(t') \otimes P_{\leq N/100} u(t'))$, we observe from \eqref{meme}, \eqref{pant} and the triangle inequality that
$$ \| P_{\leq N/100} u \|_{L^\infty_t L^{3/2}_x( [-A_3, 0] \times B(0, A_4) )} \lesssim A^2 N^{-1}.$$
Using this, \eqref{local}, \eqref{mold} (for the high frequency factor $P_{N'} u(t')$), and H\"older's inequality, we conclude that the contribution of this term to \eqref{panu} is $O( A^3 A_1^{-1} N^{-1/2})$.  Similarly for the ``low-high'' term $\tilde P_N(P_{\leq N/100} u(t') \otimes P_{N'} u(t') )$.  Finally, to control the ``high-high'' term $\sum_{N_1 \sim N_2 \gtrsim N} 
\tilde P_N(P_{N_1} u(t') \otimes P_{N_2} u(t'))$, we use \eqref{local}, the triangle inequality, H\"older, and \eqref{mold} to control this contribution by
$$ \lesssim A_4^{-40} + N^{1/2} \sum_{N_1 \sim N_2 \gtrsim N} A^3 N_1^{-2} 
 \| P_{N_2} u \|_{L^\infty_t L^{3/2}_x( [-A_3, 0] \times B(0, A_4) )}.$$
Using \eqref{meme} when $N_2 \leq A_2$ and \eqref{pant} otherwise, we see that this term also contributes $O( A^3 A_1^{-1} N^{-1/2})$.  We have thus shown that
\begin{equation}\label{thus}
  \| P_N u \|_{L^\infty_t L^2_x( [-A_3/4, 0] \times B(0, A_4/4) )} \lesssim A^3 A_1^{-1} N^{-1/2}
	\end{equation}
for $A_2^{-1/3} \leq N \leq A_2^{1/3}$.

We now return once again to Duhamel's formula to estimate
$$ |P_1 u(0,0)| \leq |e^{A_3 \Delta/4} P_1 u(-A_3/4)|(0) + \int_{-A_3/4}^0 |e^{(t-t')\Delta} P_1 \nabla \cdot \tilde P_1 (u(t') \otimes u(t'))|(0)\ dt'.$$
From \eqref{bern-2}, \eqref{able}, the first term is $O( A_3 \exp( -A_3^{2} / 320 ) )$, thus from \eqref{n1-norm} we have
$$
\int_{-A_3/4}^0 |e^{(t-t')\Delta} P_1 \nabla \cdot \tilde P_1 (u(t') \otimes u(t'))|(0)\ dt' \gtrsim A_1^{-1}.$$
From \eqref{local}, \eqref{able} one has
$$|e^{(t-t')\Delta} P_1 \nabla \cdot \tilde P_1 (u(t') \otimes u(t'))|(0) \lesssim \exp(-(t-t')/20) ( \| 
\| \tilde P_1 (u(t') \otimes u(t')) \|_{L^1_x(B(0,A_1))} + A_1^{-50} )$$
and hence by the pigeonhole principle we have
$$ \| \tilde P_1 (u(t') \otimes u(t')) \|_{L^1_x(B(0,A_1))} \gtrsim A_1^{-1}.$$
for some $-A_3/4 \leq t' \leq 0$.

Fix this $t'$.  As before, we can split $\tilde P_1 (u(t') \otimes u(t'))$ into the sum of $O(1)$ ``low-high'' terms $\tilde P_1(P_{N'} u(t') \otimes P_{\leq 1/100} u(t'))$ and ``high-low'' terms $\tilde P_1(P_{\leq 1/100} u(t') \otimes P_{N'} u(t'))$ with $N' \sim 1$, plus a ``high-high'' term $\sum_{N_1 \sim N_2 \gtrsim 1} \tilde P_1(P_{N_1} u(t') \otimes P_{N_2} u(t'))$.  For the first two types of terms, we use \eqref{local} (for frequencies larger than $A_2^{-1/3}$), \eqref{able}, and H\"older to conclude that
$$ \| P_{\leq 100} u(t') \|_{L^2_x( B(0, 2A_1) )} \lesssim A^3 A_1^{-1}$$
and then from \eqref{thus}, \eqref{local} (and \eqref{able} to control the global contribution of \eqref{local}) we see that the contribution of those two types of terms is $O( A^6 A_1^{-2})$.  For the high-high terms with $N_1,N_2 \leq A_2^{1/3}$, we again use \eqref{thus}, \eqref{local}, \eqref{able} to again obtain a bound of $O( A^6 A_1^{-2})$.  For the cases when $N_1,N_2 \gtrsim A_2^{1/3}$, we use \eqref{meme}, \eqref{able} to obtain a much better bound $O( A^3 A_2^{-1/3})$.  Putting all this together we obtain
$$ A_1^{-1} \lesssim A^6 A_1^{-2}$$
giving the required contradiction.  This establishes (iv).

Now we prove (v).  We may assume that $A_4 N_0^{-2} \leq A_4^{-1} T$, since the claim is trivial otherwise.  Thus we have $N_0 \geq A_4 T^{-1/2}$.

 By iteratively applying (iv), we may find a sequence $(t_0,x_0), (t_1,x_1), \dots, (t_n,x_n) \in [t_0-T,t_0]$ and $N_0,N_1,\dots,N_n>0$ for some $n \geq 1$, with the properties
\begin{align}
|P_{N_i} u(t_i,x_i)| &\geq A_1^{-1} N_i \label{iter-1}\\
A_2^{-1} N_{i-1} &\leq N_i \leq A_2 N_{i-1} \label{iter-2}\\
A_3^{-1} N_{i-1}^{-2} &\leq t_{i-1} - t_i \leq A_3 N_{i-1}^{-2} \label{iter-3}\\
|x_i - x_{i-1}| &\leq A_4 N_{i-1}^{-1} \label{iter-4}
\end{align}
for all $i=1,\dots,n$, with $t_i \in [t_0-T/2,t_0]$ and $N_i \geq A_3 T^{-1/2}$ for $i=0,\dots,n-1$ and either $t_n \in [t_0-T,t_0-T/2]$ or $N_n < A_3 T^{-1/2}$.  To see that this process terminates at a finite $n$, observe from the classical nature of $u$ that the $P_{N_i} u(t_i,x_i)$ are uniformly bounded in $i$, which by \eqref{iter-1} implies that the $N_i$ are uniformly bounded above, and hence by \eqref{iter-3} $t_{i-1}-t_i$ are uniformly bounded below; since $t_i$ must stay above $t_0 - T$, we obtain the required finite time termination.  By \eqref{iter-3}, the first time $t_1$ after $t_0$ lies in the interval
$$ t_1 \in [t_0 - A_2 N_0^{-2}, t_0 - A_2^{-1} N_0^{-2}].$$

If $N_n < A_3 T^{-1/2}$, then by \eqref{iter-3}, \eqref{iter-2} 
$$t_{n-1} - t_n \geq A_3^{-1} N_{n-1}^{-2} \geq A_3^{-2} N_n^{-2} \geq A_3^{-4} (t_{n-1}-t_n) \leq A_3^{-4} T$$
so in particular $t_n \leq t_0 - A_3^{-4} T$.  Of course this inequality also holds if $t_n \in [t_0-T,t_0-T/2]$.  In either case, we see from the hypothesis $A_4 N_0^{-2} \leq T_1 \leq A_4^{-1} T$ that
$$ t_n < t - T_1 \leq t_1.$$
Let $m$ be the largest index for which $t_m \geq t - T_1$, thus $1 \leq m \leq n-1$ and $t_{m+1} > t-T_1$.   By telescoping \eqref{iter-3}, we conclude that
\begin{equation}\label{aim}
 \sum_{i=0}^m A_3 N_i^{-2} = \sum_{i=1}^{m+1} A_3 N_{i-1}^{-2} \geq t-t_{m+1} \geq T_1.
\end{equation}
On the other hand, from \eqref{iter-1} and \eqref{pant} we have
$$ |P_{N_i} u(t, x_i)| \gtrsim A_1^{-1} N_i$$
for $t \in [t_i - A_1^{-2} N_i^{-2}, t_i]$; as $P_{N_i}$ is bounded on $L^\infty$ by \eqref{bern}, this implies that
$$ \| P_{N_i} u(t) \|_{L^\infty_x(\R^3)} \gtrsim A_1^{-1} N_i$$
for such $t$.  From \eqref{iter-3} we see that the time intervals $[t_i - A_1^{-2} N_i^{-2}, t_i]$ are disjoint and lie in $[t-T_1,t]$ for $i=0,\dots,m-1$.  Applying \eqref{bts}, we conclude that
$$ \sum_{i=0}^{m-1} A_1^{-1} N_i \times A_1^{-2} N_i^{-2} \lesssim A^4 T_1^{1/2} $$
and thus
$$ \sum_{i=0}^{m-1} N_i^{-1} \lesssim A_1^{4} T_1^{1/2}.$$
Using \eqref{iter-2} to extend this sum to the final index $m$, we conclude that
\begin{equation}\label{a28}
 \sum_{i=0}^{m} N_i^{-1} \lesssim A_2^{2} T_1^{1/2}.
\end{equation}
Comparing this with \eqref{aim}, we conclude that there exists $i=0,\dots,m$ such that
$$ N_i^{-1} \gtrsim A_3^{-2} T_1^{1/2}.$$
Since $A_4 N_0^{-2} \leq A_4^{-1} T$, $i$ cannot be zero, thus $1 \leq i \leq m$.
From \eqref{iter-3}, \eqref{iter-2} we have
\begin{align*}
 t_0 - t_i &\geq t_{i-1} - t_i \\
&\geq A_3^{-1} N_{i-1}^{-2}  \\
&\geq A_3^{-2} N_i^{-2} \\
&\gtrsim A_3^{-6} T_1.
\end{align*}
Since $t_0-t_i$ is also bounded by $T_1$, we also have from \eqref{iter-3} that $A_3^{-1} N_i^{-2} \leq T_1$, thus $N_i \geq A_3^{-1/2} T_1^{-1/2}$.  Finally, from telescoping \eqref{iter-4} and using \eqref{a28}, we conclude that
$$ |x_i - x_0| \lesssim A_4^{2} T_1^{1/2},$$
and the claim follows.

Finally, we prove (vi), which is the most difficult estimate.  The claim is invariant with respect to time translation and rescaling, so we may assume that $[t_0-T',t_0] = [0,1]$.  In particular $[-1,1] \subset [t_0-T,t_0]$, so we may decompose $u = u^{\lin} +u^{\nl}$ as before with the estimates \eqref{vl2}, \eqref{very}, \eqref{122}.

From \eqref{122} we can find a time $t_1 \in [-1/2,0]$ such that
$$ \int_{\R^3} |\nabla u^{\nl}(t_1,x)|^2\ dx \lesssim A^4.$$
Fix this time $t_1$.  From \eqref{very} we thus have
$$ \int_{\R^3} |\nabla u^{\nl}(t_1,x)|^2 + \sum_{j=0}^4 | \nabla^j u^{\lin}(t_1,x)|^3\ dx \lesssim A^4.$$
By the pigeonhole principle, we can thus find a scale
\begin{equation}\label{alower}
 A_6^{100} R_0 \leq R \leq \exp(A_6^{O(1)}) R_0
\end{equation}
such that
\begin{equation}\label{amp}
 \int_{A_6^{-10} R \leq |x| \leq A_6^{10} R} |\nabla u^{\nl}(t_1,x)|^2 + \sum_{j=0}^4 | \nabla^j u^{\lin}(t_1,x)|^3\ dx \lesssim A_6^{-10}.
\end{equation}
Fix this $R$.  We now propagate this estimate forward in time to $[t_1,1]$.  We first achieve this for the linear component $u^{\lin}$, which is straightforward.  From Sobolev embedding we have
$$ \sup_{A_6^{-9} R \leq |x| \leq A_6^{9} R} |\nabla^j u^{\lin}(t_1,x)| \lesssim A_6^{-3}$$
for $j=0,1,2$.  Since $\nabla^j u^{\lin}$ solves the linear heat equation, we conclude from this, \eqref{local}, and \eqref{very} that
\begin{equation}\label{linear-prop}
\sup_{t_1 \leq t \leq 1} \sup_{A_6^{-8} R \leq |x| \leq A_6^{8} R} |\nabla^j u^{\lin}(t,x)| \lesssim A_6^{-3}
\end{equation}
for $j=0,1,2$.  This estimate (when combined with \eqref{very}) will suffice to control all the terms involving the linear component $u^{\lin}$ of the velocity (or the analogous component $\omega^{\lin} \coloneqq \nabla \times u^{\lin}$ of the vorticity).

The vorticity $\omega \coloneqq \nabla \times u$ obeys the vorticity equation \eqref{vorticity}.
On $[t_1,1] \times \R^3$, we decompose $\omega = \omega^{\lin} + \omega^{\nl}$, where $\omega^{\lin} \coloneqq \nabla \times u^{\lin}$ is the linear component of the vorticity and $\omega^{\nl} \coloneqq \nabla \times u^{\nl}$ is the nonlinear component.  As $\omega^{\lin}$ solves the heat equation, we have
\begin{equation}\label{omeganl}
 \partial_t \omega^{\nl} = \Delta \omega^{\nl} - (u \cdot \nabla) \omega + (\omega \cdot \nabla) u.
\end{equation}
As in \cite[\S 10]{tao-local}, we apply the energy method to this equation with a carefully chosen time-dependent cutoff function.  Namely, let 
\begin{equation}\label{rat}
 R_- \in [A_6^{-8} R, 2A_6^{-8} R]; \quad R_+ \in [A_6^{8} R/2, A_6^{8} R]
\end{equation}
be scales to be chosen later, and define the time-dependent radii
\begin{align*}
 R_-(t) &\coloneqq R_- + C_0 \int_{t_1}^t (A_6 + \| u(t) \|_{L^\infty_x(\R^3)})\ dt \\
 R_+(t) &\coloneqq R_+ - C_0 \int_{t_1}^t (A_6 + \| u(t) \|_{L^\infty_x(\R^3)})\ dt
\end{align*}
that start at $R_-, R_+$ respectively, and contract inwards at a rate faster than the velocity field $u$.
From the bounded total speed property \eqref{bts}, \eqref{alower}, and the hypothesis $R_0 \geq 1$, we conclude that
$$ R_-(t) \in [A_6^{-8} R, 3A_6^{-8} R]; \quad R_+(t) \in [A_6^{8} R/3, A_6^{8} R]$$
for all $t \in [t_1,1]$.

For $t \in [t_1,1]$, we define the local enstrophy
$$ E(t) \coloneqq \frac{1}{2} \int_{\R^3} |\omega^{\nl}(t,x)|^2 \eta(t,x)\ dx$$
where $\eta$ is the time-varying cutoff
$$ \eta(t,x) \coloneqq \max( \min( A_6, |x| - R_-(t), R_+(t) - |x| ), 0 ),$$
thus $\eta$ is supported in the annulus $\{ R_-(t) \leq |x| \leq R_+(t) \}$, is Lipschitz with norm $1$, and equals $A_6$ in the smaller annulus $\{ R_-(t) + A_6 \leq |x| \leq R_+(t) - A_6 \}$.  From \eqref{amp} we have the initial bound
\begin{equation}\label{e-init}
 E(t_1) \lesssim A_6^{-9}.
\end{equation}
Now we control the time derivative $\partial_t E(t)$ for $t \in [t_1,1]$.  From \eqref{omeganl} and integration by parts we have
$$ \partial_t E(t) = -Y_1(t) - Y_2(t) + Y_3(t) + Y_4(t) + Y_5(t) + Y_6(t) + Y_7(t) + Y_8(t) + Y_9(t)$$
where $Y_1$ is the dissipation term
$$ Y_1(t) \coloneqq \int_{\R^3} |\nabla \omega^{\nl}(t,x)|^2\ dx,$$
$Y_2(t)$ is the recession term
$$ Y_2(t) \coloneqq -\frac{1}{2} \int_{\R^3} |\omega^{\nl}(t,x)|^2 \partial_t \eta(t,x)\ dx,$$
$Y_3(t)$ is the heat flux term
$$ Y_3(t) \coloneqq \frac{1}{2} \int_{\R^3} |\omega^{\nl}(t,x)|^2 \Delta \eta(t,x)\ dx, $$
$Y_4(t)$ is the transport term
$$ Y_4(t) \coloneqq \frac{1}{2} \int_{\R^3} |\omega^{\nl}(t,x)|^2 u(t,x) \cdot \nabla \eta(t,x)\ dx,$$
$Y_5(t)$ is a correction to the transport term arising from $\omega^{\lin}$,
$$ Y_5(t) \coloneqq - \int_{\R^3} \omega^{\nl}(t,x) \cdot (u(t,x) \cdot \nabla) \omega^{\lin}(t,x)\ \eta(t,x)\ dx,$$
$Y_6(t)$ is the main nonlinear term
$$ Y_6(t) \coloneqq \int_{\R^3} \omega^{\nl}(t,x) \cdot (\omega^{\nl}(t,x) \cdot \nabla) u^{\nl}(t,x)\ \eta(t,x)\ dx$$
and $Y_7(t),Y_8(t),Y_9(t)$ are corrections to the transport term arising from the $u^{\lin}$ and $\omega^{\lin}$,
\begin{align*}
Y_7(t) &\coloneqq \int_{\R^3} \omega^{\nl}(t,x) \cdot (\omega^{\nl}(t,x) \cdot \nabla) u^{\lin}(t,x)\ \eta(t,x)\ dx\\
Y_8(t) &\coloneqq \int_{\R^3} \omega^{\nl}(t,x) \cdot (\omega^{\lin}(t,x) \cdot \nabla) u^{\nl}(t,x)\ \eta(t,x)\ dx\\
Y_9(t) &\coloneqq \int_{\R^3} \omega^{\nl}(t,x) \cdot (\omega^{\lin}(t,x) \cdot \nabla) u^{\lin}(t,x)\ \eta(t,x)\ dx.
\end{align*}
Here all derivatives of the Lipschitz function $\eta$ are interpreted in a distributional sense.
We now aim to control $Y_3(t),\dots,Y_9(t)$ in terms of $Y_1(t), Y_2(t), E(t)$, and some other quantities that are well controlled.  From definition of $\eta$ we see that 
$$-\partial_t \eta(t,x) = C_0 (A_6 + \| u(t) \|_{L^\infty_x(\R^3)}) |\nabla \eta(t,x)|$$
so in particular we have that $Y_2(t)$ is non-negative and
$$ Y_4(t) \leq C_0^{-1} Y_2(t).$$
A direct computation of $\Delta \eta$ in polar coordinates yields the bound
\begin{align*}
 Y_3(t) &\lesssim \int_{|x| \in [R_-(t),R_-(t)+A_6] \cup [R_+(t)-A_6,R_+(t)]} \frac{|\omega^{\nl}(t,x)|^2}{|x|}\ dx  \\
&\quad + \sum_{r = R_-(t), R_-(t)+A_6, R_+(t)-A_6, R_+(t)} r^2 \int_{S^2} |\omega^{\nl}(t,r\theta)|\ d\theta
\end{align*}
where $d\theta$ is surface measure on the sphere (in fact the $r = R_-(t)+A_6, R_+(t)-A_6$ terms are non-positive and could be discarded if desired).  This expression is difficult to estimate for fixed choices of $R_-, R_+$.  However, if selects $R_-, R_+$ uniformly at random from the range \eqref{rat}, we see from Fubini's theorem that the expected value $\E |Y_3|$ of $|Y_3|$ can be estimated by
$$ \E |Y_3(t)| \lesssim A_6 \int_{|x| \in 
[A_6^{-8} R, 3A_6^{-8} R] \cup [A_6^{8} R/3, A_6^{8} R]} \frac{|\omega^{\nl}(t,x)|^2}{|x|^2}\ dx $$
and hence by \eqref{122}, \eqref{alower}
$$ \E \int_{t_1}^1 |Y_3(t)|\ dt \lesssim A_6^{-10}$$
(say).  Thus we can select $R_-, R_+$ so that
\begin{equation}\label{yat}
 \int_{t_1}^1 |Y_3(t)|\ dt \lesssim A_6^{-10}
\end{equation}
and we shall now do so.

To treat $Y_5(t)$, we use Young's inequality to bound
$$ Y_5(t) \lesssim E(t) + \int_{\R^3} |(u \cdot \nabla) \omega^{\lin}|^2\ \eta\ dx.$$
Using \eqref{able}, \eqref{very}, \eqref{linear-prop}, H\"older's inequality, we then have
$$ Y_5(t) \lesssim E(t) + A_6^{-2}$$
(say).

In a similar vein, from \eqref{linear-prop} and H\"older's inequality one has
$$ Y_7(t) \lesssim E(t)$$
(with plenty of room to spare) and from Young's inequality one has
$$ Y_9(t) \lesssim E(t) + \int_{\R^3} |(\omega^{\lin} \cdot \nabla) u^{\lin}|^2\ \eta(t,x)\ dx$$
and hence by \eqref{able}, \eqref{very}, \eqref{linear-prop}, and H\"older
$$ Y_9(t) \lesssim E(t) + A_6^{-2}.$$
For $Y_8$, we again use Young's inequality to bound
$$Y_8(t) \lesssim E(t) + \int_{\R^3} |(\omega^{\lin} \cdot \nabla) u^{\nl}|^2\ \eta(t,x)\ dx$$
and hence by \eqref{linear-prop}
$$Y_8(t) \lesssim E(t) + Y_{10}(t)$$
where
$$ Y_{10}(t) \coloneqq A_6^{-3} \int_{\R^3} |\nabla u^{\nl}(t,x)|^2\ dx.$$
Observe from \eqref{122} that
\begin{equation}\label{yat-2}
 \int_{t_1}^1 |Y_{10}(t)|\ dt \lesssim A_6^{-2}.
\end{equation}

We are left with estimation of the most difficult term $Y_6(t)$.  Following \cite{tao-local}, we cover the annulus $\{ R_-(t) \leq |x| \leq R_+(t) \}$ by a boundedly overlapping Whitney decomposition of balls $B = B( x_B, r_B )$, where the radius $r_B$ of the ball is given as $r_B \coloneqq \frac{1}{100} \eta(t,r_B)$.  In particular, we have $\eta(t,x) \sim r_B$ on the dilate $10B = B(x_B, r_B)$ of the ball.  We can then write
$$ Y_6(t) \sim \sum_B r_B \int_B |\omega^{\nl}|^2 |\nabla u^{\nl}|\ dx$$
where we suppress the explicit dependence on $t,x$ for brevity.  Similarly one has
\begin{equation}\label{eb}
 E(t) \sim \sum_B r_B \int_{10B} |\omega^{\nl}|^2\ dx
\end{equation}
and
\begin{equation}\label{yatta}
 Y_1(t) \sim \sum_B r_B \int_{10B} |\nabla \omega^{\nl}|^2\ dx
\end{equation}
To control $Y_6(t)$, we need to control $\nabla u^{\nl}$.  The Biot-Savart law suggests that this function has comparable size to $\omega^{\nl}$, but we need to localise this intuition to the ball $B$ and thus must address the slightly non-local nature of the Biot-Savart law.  Fortunately this can be handled using standard cutoff functions.  Namely, we have $\Delta u^{\nl} = - \nabla \times \omega^{\nl}$, hence if we let $\psi_B$ be a smooth cutoff adapted to $3B$ that equals $1$ on $2B$, then
$$ u^{\nl} = - \Delta^{-1} (\nabla \times (\omega^{\nl}\psi_B)) + v$$
where $v$ is harmonic on $2B$.  From Sobolev embedding and H\"older one has
$$ \| v\|_{L^2_x(2B)} \lesssim \| \omega^{\nl} \psi_B \|_{L^{6/5}_x(\R^3)} + \| u^{\nl} \|_{L^2_x(2B)}  
\lesssim r_B^{3/2} \| \omega^{\nl} \|_{L^3_x(3B)} + \| u^{\nl} \|_{L^2_x(2B)} $$
and hence by elliptic regularity for harmonic functions
$$ \| \nabla v \|_{L^\infty_x(B)} \lesssim r_B^{-5/2} \| v\|_{L^2_x(2B)} 
\lesssim r_B^{-1} \| \omega^{\nl} \|_{L^3_x(3B)} + r_B^{-5/2} \| u^{\nl} \|_{L^2_x(2B)}.$$
We conclude the pointwise estimate
\begin{equation}\label{nab}
 \nabla u^{\nl} = - \nabla \Delta^{-1} (\nabla \times (\omega^{\nl}\psi_B)) +
O( r_B^{-1} \| \omega^{\nl} \|_{L^3_x(3B)} ) + O( r_B^{-5/2} \| u^{\nl} \|_{L^2_x(2B)} )
\end{equation}
on $B$.  By elliptic regularity, $\nabla \Delta^{-1} (\nabla \times (\omega^{\nl}\psi_B)) $ has an $L^3_x(B)$ norm of 
$O( \| \omega^{\nl} \|_{L^3_x(3B)} )$.  From H\"older's inequality we thus have
$$ \int_B |\omega^{\nl}|^2 |\nabla u^{\nl}|\ dx \lesssim \| \omega^{\nl} \|_{L^3_x(3B)}^3 + r_B^{-5/2} \| \omega^{\nl} \|_{L^2_x(3B)}^2 
\| u^{\nl} \|_{L^2_x(3B)} $$
and hence $Y_6(t) \lesssim Y_{6,1}(t) + Y_{6,2}(t)$, where
$$ Y_{6,1}(t) \coloneqq \sum_B r_B \| \omega^{\nl} \|_{L^3_x(3B)}^3 $$
and
$$ Y_{6,2}(t) \coloneqq \sum_B r_B^{-3/2} \| \omega^{\nl} \|_{L^2_x(3B)}^2 \| u^{\nl} \|_{L^2_x(3B)}.$$
For $Y_{6,2}(t)$, we first consider the contribution of the large balls in which $r_B \geq A^{10}$.  Here we simply use \eqref{vl2} to bound $\| u^{\nl} \|_{L^2_x(3B)} \lesssim A^2$.  Since $r_B^{-3/2} A^2 \lesssim r_B$ for large balls $B$, the contribution of this case is $O( E(t))$ thanks to \eqref{eb}.  Now we look at the small balls in which $r_B < A^{10}$.  Here we use H\"older to bound
$$ \| u^{\nl} \|_{L^2_x(3B)} \lesssim r_B^{3/2} \| u^{\nl} \|_{L^\infty_x(\R^3)} \lesssim r_B^{3/2} ( A^2 + \| u \|_{L^\infty_x(\R^3)} )$$
so the contribution of this case is bounded by
$$ \sum_{B: r_B < A^{10}} \| \omega^{\nl} \|_{L^2_x(3B)}^2 ( A^2 + \| u \|_{L^\infty_x(\R^3)} ).$$
For small balls $B$, $3B$ is completely contained inside the region in which $\partial_t \eta \gtrsim C_0( A_6 + \| u \|_{L^\infty_x(\R^3)})$, so the contribution of this case can be bounded by $O( C_0^{-1} Y_2(t) )$.  Thus
$$ Y_{6,2}(t) \lesssim E(t) + C_0^{-1} Y_2(t).$$
Now we control $Y_{6,1}(t)$.  For each ball $B$, define the mean vorticity $\omega_B$ by
$$ \omega_B \coloneqq \frac{\int_{\R^3} \omega^{\nl} \psi_B\ dx}{\int_{\R^3} \psi_B\ dx}.$$
From the Poincar\'e inequality, Sobolev embedding, and the triangle inequality we have
$$ \| \omega^{\nl} - \omega_B \|_{L^6_x(3B)} \lesssim \| \nabla \omega^{\nl} \|_{L^2_x(10B)}$$
and similarly
\begin{equation}\label{po}
 |\omega_B - \omega_{B'}| \lesssim r_B^{-1/2} \| \nabla \omega^{\nl} \|_{L^2_x(10B)}
\end{equation}
when $B,B'$ are overlapping Whitney balls.  We can now use H\"older's inequality to bound
\begin{align*}
 Y_{6,1}(t) &\lesssim \sum_B r_B^4 \omega_B^3 + \sum_B r_B \| \omega^{\nl} - \omega_B \|_{L^3_x(3B)}^3  \\
&\lesssim \sum_B r_B^4 \omega_B^3 + \sum_B r_B \| \omega^{\nl} - \omega_B \|_{L^2_x(3B)}^{3/2} \| \nabla \omega^{\nl} \|_{L^2_x(10B)}^{3/2}.
\end{align*}
By Young's inequality and \eqref{yatta}, we then have
$$ Y_{6,1}(t) \leq \frac{1}{2} Y_1(t) + O( \sum_B r_B^4 \omega_B^3 + \sum_B r_B \| \omega^{\nl} - \omega_B \|_{L^2_x(3B)}^6 ).$$
From the triangle inequality and Cauchy-Schwarz, and \eqref{eb}, one has
$$ \| \omega^{\nl} - \omega_B \|_{L^2_x(3B)} \lesssim \| \omega^{\nl} \|_{L^2_x(3B)} \lesssim r_B^{-1/2} E(t)^{1/2}
$$
and hence from H\"older, \eqref{yatta}
$$\sum_B r_B \| \omega^{\nl} - \omega_B \|_{L^2_x(3B)}^6  \lesssim \sum_B r_B\| \omega^{\nl} - \omega_B \|_{L^6_x(3B)}^2 E(t)^2
\lesssim E(t)^2 Y_1(t).$$
Now we estimate $\sum_B r_B^4 \omega_B^3$.  We can arrange the Whitney decomposition so that all the radii $r_B$ are powers of $1.001$, and that every ball $B$ of radius less than (say) $A_6 / 100$ has a ``parent'' ball $p(B)$ that overlaps $B$ and has radius $1.001 r_B$.  	From the triangle inequality we have
$$ |\omega_B| \leq |\omega_{p^k(B)}| + \sum_{i=0}^{k-1} |\omega_{p^i(B)} - \omega_{p^{i+1}(B)}|$$
for any Whitney ball $B$, where $k = k_B$ is the first natural number for which the iterated parent $p^k(B)$ has radius larger than $A_6^{1/2}$.  By H\"older we then have
$$ |\omega_B|^3 \lesssim |\omega_{p^k(B)}|^3 + \sum_{i=0}^{k-1} (1+i)^{10} |\omega_{p^i(B)} - \omega_{p^{i+1}(B)}|^3.$$
From a volume packing argment we see that for a given $i$, a Whitney ball $B'$ is of the form $p^i(B)$ for at most $O( (1.001)^{2i})$ choices of $B$.  One can then sum the geometric series (exactly as in \cite[\S 10]{tao-local}) and conclude that
$$ \sum_B r_B^4 \omega_B^3 \lesssim \sum_{B: r_B \geq A_6/100} r_B^4 \omega_B^3 + \sum_{B: r_B < A_6^{1/2}} r_B^4 |\omega_B - \omega_{p(B)}|^3.$$
For the small balls in which $r_B < A_6/100$, we observe from \eqref{yatta} and Cauchy-Schwarz that
$$ \omega_B, \omega_{p(B)} \lesssim r_B^{-2} E(t)^{1/2}$$
and thus from \eqref{po}, \eqref{yatta}
$$ \sum_{B: r_B < A_6/100} r_B^4 |\omega_B - \omega_{p(B)}|^3 \lesssim E(t)^{1/2} Y_1(t).$$
For the large balls in which $r_B \geq A_6/100$, we write $\omega^{\nl} = \nabla \times u^{\nl}$ and integrate by parts using Cauchy-Schwarz to find that
$$ \omega_B \lesssim r_B^{-5/2} \| u^{\nl} \|_{L^2_x(3B)} $$
and hence using \eqref{vl2} and the bounded overlap of the Whitney balls
$$ \sum_{B: r_B \geq A_6/100} r_B^4 \omega_B^3 \lesssim \sum_{B: r_B \geq A_6/100} A^2 r_B^{-7/2} \| u^{\nl} \|_{L^2_x(3B)}^2
\lesssim A^4 A_6^{-7/2} \lesssim A_6^{-2}.$$ 
Thus we have
$$ Y_{6,1} \leq \frac{1}{2} Y_1(t) + O(  E(t)^{1/2} Y_1(t) + A_6^{-2} + E(t)^2 Y_1(t) ).$$

Putting all this together, we see that
$$ \partial_t E(t) \leq - \frac{1}{2} Y_1(t) + O\left( E(t) + |Y_3(t)| + |Y_{10}(t)| + A_6^{-2} + E(t)^{1/2} Y_1(t) + E(t)^2 Y_1(t)\right ).$$
A standard continuity argument using \eqref{e-init}, \eqref{yat}, \eqref{yat-2} then gives
\begin{equation}\label{eo}
 E(t) \lesssim A_6^{-2}
\end{equation}
for all $t_1 \leq t \leq 1$, and also
\begin{equation}\label{eo-2}
\int_{t_1}^1 Y_1(t)\ dt \lesssim A_6^{-2}.
\end{equation}
These are subcritical regularity estimates and can now be iterated\footnote{It is likely that one can also proceed at this point using the local regularity theory from \cite{ckn}.} as in the proof of (iii) to obtain higher regularity.  First we move from control of the vorticity back to control of the velocity.  From \eqref{nab} and elliptic regularity one has
$$
\| \nabla u^{\nl} \|_{L^2_x(B)}^2 \lesssim \| \omega^{\nl} \|_{L^2_x(3B)}^2 + r_B^{-2} \| u^{\nl} \|_{L^2_x(2B)}^2 $$
for any ball $B$; summing this on balls of radius $A_6^{10}$ (say) using \eqref{eo}, \eqref{vl2}, we conclude that
\begin{equation}\label{dance}
\int_{A_6^{-7} R \leq |x| \leq A_6^{7} R} |\nabla u^{\nl}(t,x)|^2\ dx \lesssim A_6^{-2}
\end{equation}
for all $t_1 \leq t \leq 1$.  Similarly we have
$$
\| \nabla^2 u^{\nl} \|_{L^2_x(B)}^2 \lesssim \| \nabla \omega^{\nl} \|_{L^2_x(3B)}^2 + r_B^{-2} \| \nabla u^{\nl} \|_{L^2_x(2B)}^2 $$
and using \eqref{dance}, \eqref{eo-2} in place of \eqref{vl2}, \eqref{eo} we conclude that
$$
\int_{t_1}^1 \int_{A_6^{-6} R \leq |x| \leq A_6^{6} R} |\nabla^2 u^{\nl}(t,x)|^2\ dx dt \lesssim A_6^{-2}.$$
Using the Gagliardo-Nirenberg inequality \eqref{unl} as before we see that
$$
\| u^{\nl} \|_{L^4_t L^\infty_x( [t_1,1] \times \{ A_6^{-5} R \leq |x| \leq A_6^{5} R \})} \lesssim A_6^{-2}$$
which when combined with \eqref{linear-prop} gives
$$
\| u \|_{L^4_t L^\infty_x( [t_1,1] \times \{ A_6^{-5} R \leq |x| \leq A_6^{5} R \})} \lesssim A_6^{-2}.$$
By repeating the arguments in (iii) (using \eqref{local} in place of \eqref{global} to handle the long-range components of the heat kernel, which can be controlled with extremely good bounds using \eqref{able}), one can then show iteratively that
$$
\| u \|_{L^8_t L^\infty_x( [t_1,1] \times \{ A_6^{-4} R \leq |x| \leq A_6^{4} R \})} \lesssim A_6^{-2},$$
then
$$
\| u \|_{L^\infty_t L^\infty_x( [t_1,1] \times \{ A_6^{-3} R \leq |x| \leq A_6^3 R \})} \lesssim A_6^{-2},$$
then
$$
\| \nabla u \|_{L^4_t L^\infty_x( [t_1,1] \times \{ A_6^{-2} R \leq |x| \leq A_6^2 R \})} \lesssim A_6^{-2},$$
then finally
$$
\| \nabla u \|_{L^\infty_t L^\infty_x( [t_1,1] \times \{ A_6^{-1} R \leq |x| \leq 2 A_6 R \})} \lesssim A_6^{-2}$$
and
$$
\| \nabla \omega \|_{L^\infty_t L^\infty_x( [t_1,1] \times \{ R \leq |x| \leq A_6 R \})} \lesssim A_6^{-2}$$
giving (vi).

\section{Carleman inequalities for backwards heat equations}\label{carleman-sec}

We will need some Carleman inequalities for backwards heat equations which are essentially contained in previous literature (most notably \cite{ess}, \cite{ess-back}), but made slightly more quantitative for our application (also it will be convenient to not demand that the functions involved vanish at the starting and final time).  Following \cite{ess}, we shall reverse the direction of time and work here with backwards heat equations rather than forward ones.

Our main tool is the following general inequality (cf. \cite[Lemma 2]{ess-back}):

\begin{lemma}[General Carleman inequality]\label{carl}  Let $[t_1,t_2]$ be a time interval, and let $u \in C^\infty_c( [t_1,t_2] \times \R^d \to \R^m)$ be a (vector-valued) test function solving the backwards heat equation
$$ L u = f$$
with $L$ the backwards heat operator
\begin{equation}\label{L-def}
L \coloneqq \partial_t + \Delta,
\end{equation}
and let $g: [t_1,t_2] \times \R^d \to \R$ be smooth.  Let $F: [t_1,t_2] \times \R^d \to \R$ denote the function
$$ F \coloneqq \partial_t g - \Delta g - |\nabla g|^2.$$
Then we have the inequality
$$ \partial_t \int_{\R^d} \left(|\nabla u|^2 + \frac{1}{2} F |u|^2\right)\ e^g dx
\geq \int_{\R^d} \left( \frac{1}{2} (LF) |u|^2 + 2D^2 g(\nabla u, \nabla u) - \frac{1}{2} |Lu|^2 \right)\ e^g dx$$
for all $t \in I$, where $D^2 g$ is the bilinear form expressed in coordinates as
$$ D^2 g( v, w ) \coloneqq (\partial_i \partial_j g) v_i \cdot w_j$$
with the usual summation conventions.  In particular, from the fundamental theorem of calculus one has
\begin{align*}
& \int_{t_1}^{t_2} \int_{\R^d} \left( \frac{1}{2} (LF) |u|^2 + 2D^2 g(\nabla u, \nabla u) \right)\ e^g dx dt \\
&\quad \leq \frac{1}{2} \int_{t_1}^{t_2} \int_{\R^d} |Lu|^2\ e^g dx dt + \int_{\R^d} \left(|\nabla u|^2 + \frac{1}{2} F |u|^2\right)\ e^g dx|^{t=t_2}_{t=t_1}.
\end{align*}
\end{lemma}

The above inequality is valid in all dimensions, but in this paper we will only need this lemma in the case $d=m=3$.

\begin{proof}  By breaking $u$ into components, we may assume without loss of generality that we are in the scalar case $m=1$.

We use the usual commutator method.  Introducing the weighted (and time-dependent) inner product
$$ \langle u, v \rangle \coloneqq \int_{\R^n} u v\ e^g dx$$
for test functions $u,v: I \times \R^n \to \R$, we compute after differentiating under the integral sign and integrating by parts
\begin{align*}
\langle Lu, v \rangle + \langle u, Lv \rangle &= \int_{\R^n} \left( \partial_t(uv) + \Delta(u v) - 2 \nabla u \cdot \nabla v \right)\ e^g dx \\
&= \partial_t \langle u,v \rangle + \int_{\R^n} \left( - (\partial_t g) uv + (\Delta g + |\nabla g|^2) uv - 2 \nabla u \cdot \nabla v \right)\ e^g dx \\
&= \partial_t \langle u,v \rangle - \langle F u, v \rangle - 2 \langle \partial_i u, \partial_i v \rangle 
\end{align*}
with the usual summation conventions.  We can write
\begin{equation}\label{amy}
 - \langle \partial_i u, \partial_i v \rangle - \frac{1}{2} \langle F u, v \rangle = \langle Su , v \rangle
\end{equation}
where $S$ is the differential operator
$$ S u \coloneqq \Delta u + \nabla g \cdot \nabla u - \frac{1}{2} F u$$
which is then formally self-adjoint with respect to the inner product $\langle,\rangle$; one can view $S$ as the self-adjoint component of $L$.  We can then rewrite the above identity as
$$ \partial_t \langle u, v \rangle = \langle Lu, v \rangle + \langle u, Lv \rangle - 2 \langle Su, v \rangle.$$
In particular, by the self-adjointness of $S$ we have for any test functions $u,v$ that
\begin{align*}
\partial_t \langle Su, v \rangle &= \langle LSu, v \rangle + \langle Su, Lv \rangle - 2 \langle Su, Sv \rangle \\
&= \langle [L,S]u, v \rangle + \langle SLu, v \rangle + \langle Su, Lv \rangle - 2 \langle Su, Sv \rangle \\
&= \langle [L,S]u, v \rangle + \langle Lu, Sv \rangle + \langle Su, Lv \rangle - 2 \langle Su, Sv \rangle \\
&= \langle [L,S]u, v \rangle + \frac{1}{2} \langle Lu, Lv \rangle - \frac{1}{2} \langle (L-2S)u, (L-2S)v\rangle.
\end{align*}
Among other things, this shows that the differential operator $[L,S]$ (which does not involve any time derivatives) is formally self-adjoint with respect to the inner product $\langle,\rangle$.  Specialising to the case $u=v$, we conclude in particular the inequality
\begin{equation}\label{sa}
\partial_t \langle Su, u \rangle \leq \langle [L,S]u, u \rangle + \frac{1}{2} \langle Lu, Lu \rangle.
\end{equation}
Now we compute $[L,S]$.  As previously noted, $[L,S]$ is a formally self-adjoint differential operator that does not involve any time derivatives.  Since the second order operator $L$ commutes with the second order component $\Delta$ of $S$, we see that $[L,S]$ is a second-order operator.  The highest order terms can be easily computed in coordinates as
$$ [L,S] u = 2 (\partial_i \partial_j g) \partial_i \partial_j u + \mathrm{l.o.t.}$$
and hence after integrating by parts the symmetric quadratic form $\langle [L,S] u, v \rangle$ must take the form
$$ \langle [L,S] u, v \rangle = \int_{\R^d} ( -2 D^2 g( \nabla u, \nabla v ) + H u v )\ e^g dx$$
for some function $H$; setting $u=1$, we see that $H$ must equal 
$$ H = [L,S] 1 = LS1 = - \frac{1}{2} LF.$$
We conclude that
$$ \langle [L,S] u, u \rangle = \int_{\R^d} ( - 2D^2 g( \nabla u, \nabla u ) - \frac{1}{2} (LF) |u|^2 )\ e^g dx.$$
Inserting this identity back into \eqref{sa} and using \eqref{amy}, we obtain the claim.
\end{proof}

The inequality below is a quantitative variant of \cite[Lemma 4]{ess-back}.

\begin{proposition}[First Carleman inequality]\label{carl-first}  Let $T>0$, $0 < r_- < r_+$, and let ${\mathcal A}$ denote the cylindrical annulus
$$ {\mathcal A} \coloneqq \{ (t,x) \in \R \times \R^3: t \in [0,T]; r_- \leq |x| \leq r_+ \}.$$
Let $u: {\mathcal A} \to \R^3$ be a smooth function obeying the differential inequality
\begin{equation}\label{lu}
 |Lu| \leq C_0^{-1} T^{-1} |u| + C_0^{-1/2} T^{-1/2} |\nabla u|
\end{equation}
on ${\mathcal A}$.  Assume the inequality
\begin{align}
r_-^2 &\geq 4 C_0 T. \label{sigma-2}
\end{align}
Then one has
$$
\int_{0}^{T/4} \int_{10r_- \leq |x| \leq r_+/2} \left( T^{-1} |u|^2 + |\nabla u|^2\right)\ dx dt \lesssim C_0^2 e^{-\frac{r_- r_+}{4C_0 T}} (X + e^{2 r_+^2 / C_0 T} Y)$$
where
$$
X \coloneqq \int\int_{\mathcal A} e^{2 |x|^2 / C_0 T} (T^{-1} |u(t,x)|^2 + |\nabla u(t,x)|^2)\ dx dt
$$
and
$$ Y \coloneqq \int_{r_- \leq |x| \leq r_+} |u(0,x)|^2\ dx.$$
\end{proposition}

The key feature here is the gain of $e^{-\frac{r_- r_+}{4C_0T}}$, which can be compared against the trivial bound of $e^{-2 r_-^2 / C_0T} X$ that follows by lower bounding the factor $e^{2|x|^2/C_0T}$ appearing in $X$ by $e^{2r_-^2/C_0T}$.  Thus, this lemma becomes powerful when the ratio $r_+/r_-$ is large.  Informally, Proposition \ref{carl-first} asserts that if $u$ solves \eqref{lu} on $\mathcal{A}$, has some mild Gaussian decay as $|x| \to \infty$, and is extremely small at $t=0$, then it is also very small in the interior of $\mathcal{A}$ near $t=0$.  The various numerical constants such as $1/4$ or $10$ appearing in the above proposition can be modified (and optimised) if desired, but we fix a specific choice of constants for sake of concreteness.  The weight $e^{2|x|^2/C_0T}$ in $X$ is inconvenient, but it is negligible when compared against the ``natural'' decay rate of $e^{-|x|^2/4t}$ arising from the fundamental solution of the heat equation, and it can be managed in our application by using the second Carleman inequality given below.  Specialising Proposition \ref{carl-first} the case $u(0,x)=0$ (so that $Y=0$) and sending $r_+$ to infinity, one recovers a variant of the backwards uniqueness result in \cite[Lemma 4]{ess-back}.

\begin{proof}  We may assume that
\begin{equation}\label{sigma-3}
r_+ \geq 20 r_-
\end{equation}
since the claim is vacuous otherwise.
By the pigeonhole principle, one can find a time $T_0 \in [T/2,T]$ such that
\begin{equation}\label{quo}
 \int_{r_- \leq |x| \leq r_+} e^{2 |x|^2/C_0 T} (T^{-1} |u(T_0,x)|^2 + |\nabla u(T_0,x)|^2)\ dx \lesssim T^{-1} X.
\end{equation}
Fix this time $T_0$.  In the discussion below we implicitly restrict $(t,x)$ to the region $\mathcal{A} \cap ([0,T_0] \times \R^3)$.  We set 
\begin{equation}\label{alpha-def}
\alpha \coloneqq \frac{r_+}{2C_0 T^2}
\end{equation}
and observe from \eqref{sigma-2}, \eqref{sigma-3} that
\begin{equation}\label{alphar}
\alpha \geq \frac{40}{r_- T}.
\end{equation}

Following \cite{ess}, we apply Lemma \ref{carl} on the interval $[0,T_0]$ with the weight
$$ g \coloneqq \alpha (T_0-t) |x| + \frac{1}{C_0 T} |x|^2$$
and $u$ replaced by $\psi u$, where $\psi(x)$ is a smooth cutoff supported on the region $r_1 \leq |x| \leq r_2$ that equals $1$ on $2r_1 \leq |x| \leq r_2/2$ and obeys the estimates $|\nabla^j \psi(x)| = O( 1/|x|^j)$ for $j=0,1,2$.  Since $\alpha (T_0-t)|x|$ is convex in $x$, we have
$$ D^2 g(\nabla(\psi u), \nabla(\psi u)) \geq 2 \frac{1}{C_0 T} |\nabla(\psi u)|^2.$$
The function $F$ defined in Lemma \ref{carl} can be computed on ${\mathcal A}$ as
\begin{align*}
 F &= - \alpha |x| -\frac{2\alpha (T_0-t)}{|x|} - \frac{6}{C_0 T} - \left( \alpha(T_0-t) + \frac{2 |x|}{C_0 T} \right)^2 \\
 &= - \alpha |x| -\frac{2\alpha (T_0-t)}{|x|} - \frac{6}{C_0 T} - \alpha^2(T_0-t)^2 - 4 \alpha \frac{T_0-t}{C_0 T} |x| - \frac{4 |x|^2}{C_0^2 T^2}.
\end{align*}
In particular $F$ is negative.  We also calculate
$$
LF = 2 \alpha^2(T_0-t) + \frac{4 \alpha |x|}{C_0 T} - \frac{8 \alpha (T_0-t)}{C_0 T |x|} - \frac{24}{C_0^2 T^2}.$$
We see from \eqref{sigma-2} that $\frac{T_0-t}{|x|} \leq \frac{1}{4} |x|$, so that
$$ LF \geq \frac{2 \alpha |x|}{C_0 T} - \frac{24}{C_0^2 T^2}.$$
By \eqref{alphar} we thus have $LF \geq \frac{56}{C_0^2 T^2}$.  Applying Lemma \ref{carl} and discarding some terms, we conclude that
\begin{align*}
& \int_{0}^{T_0} \int_{2r_- \leq |x| \leq r_+/2} \left( 28 C_0^{-2} T^{-2} |u|^2 + 4 C_0^{-1} T^{-1} |\nabla u|^2\right)\ e^g dx dt \\
&\quad \leq \frac{1}{2} \int_{0}^{T_0} \int_{\R^3} |L(\psi u)|^2\ e^g dx dt \\
&\quad\quad + \int_{\R^3} |\nabla(\psi u)(T_0,x)|^2\ e^{g(T_0,x)} dx + \int_{\R^3} |(\psi u)(0,x)|^2 |F(0,x)|\ e^{g(0,x)}\ dx.
\end{align*}
In the region  $2r_- \leq |x| \leq r_+/2$ we have from \eqref{lu} that
$$ |L(\psi u)|^2 = |Lu|^2 \leq 2 C_0^{-2} T^{-2} |u|^2 + 2 C_0^{-1} T^{-1} |\nabla u|^2.$$
In the regions $r_- \leq |x| \leq 2r_-$ or $r_+/2 \leq |x| \leq r_+$, we have
$$ |L(\psi u)|^2 \lesssim |Lu|^2 + |x|^{-2} |\nabla u|^2 + |x|^{-4} |u|^2 \lesssim
C_0^{-2} T^{-2} |u|^2 + C_0^{-1} T^{-1} |\nabla u|^2$$
thanks to \eqref{lu}, \eqref{sigma-2}.  For all other $x$, $L(\psi u)$ vanishes.  A similar calculation gives 
$$ |\nabla(\psi u)|^2 \lesssim |\nabla u|^2 + C_0^{-1} T^{-1} |u|^2.$$
We therefore have
\begin{align*}
& \int_{0}^{T_0} \int_{2r_- \leq |x| \leq r_+/2} \left( C_0^{-2} T^{-2} |u|^2 + C_0^{-1} T^{-1} |\nabla u|^2\right)\ e^g dx dt \\
&\quad \lesssim \int_{0}^{T_0} \int_{|x| \in [r_-,2r_-] \cup [r_+/2,r_+]} (C_0^{-2} T^{-2} |u|^2 + C_0^{-1} T^{-1} |\nabla u|^2)\ e^g dx dt \\
&\quad\quad + \int_{r_- \leq |x| \leq r_+} (C_0^{-1} T^{-1} |u|^2 + |\nabla u|^2)(T_0,x)\ e^{g(T_0,x)} dx + \int_{r_- \leq |x| \leq r_+} |u(0,x)|^2 |F(0,x)|\ e^{g(0,x)}\ dx.
\end{align*}
From \eqref{quo} one has
$$ \int_{r_- \leq |x| \leq r_+} (C_0^{-1} T^{-1} |u|^2 + |\nabla u|^2)(T_0,x)\ e^{g(T_0,x)} dx  \lesssim T^{-1} X.$$
When $t \in [0,T_0]$ and $|x| \in [r_+/2,r_+]$, one has
$$ e^g \leq e^{\alpha T |x| - \frac{|x|^2}{C_0 T}} e^{2 |x|^2 / C_0T} \leq e^{2 |x|^2 / C_0 T}$$
by \eqref{alpha-def}.  When instead $t \in [0,T_0]$ and $|x| \in [r_-,2r_-]$, one has
$$ e^g \leq e^{\alpha T |x| - \frac{|x|^2}{C_0 T}} e^{2 |x|^2 / C_0 T} \leq e^{2\alpha T r_-} e^{2 |x|^2 / C_0 T}.$$
We conclude that
\begin{align*}
& \int_{0}^{T_0} \int_{2r_- \leq |x| \leq r_+/2} \left( C_0^{-2} T^{-2} |u|^2 + C_0^{-1} T^{-1} |\nabla u|^2\right)\ e^g dx dt \\
&\quad \lesssim e^{2\alpha T r_-} T^{-1} X + \int_{r_- \leq |x| \leq r_+} |u(0,x)|^2 |F(0,x)|\ e^{g(0,x)}\ dx.
\end{align*}
In the region $t \in [0,T/4]$, $10r_- \leq |x| \leq r_+/2$, one has
$$ e^g \geq e^{\frac{\alpha T |x|}{4} + \frac{|x|^2}{C_0 T}} \geq e^{\frac{5}{2} \alpha T r_-}$$
and hence
\begin{align*}
& \int_{0}^{T/4} \int_{10r_- \leq |x| \leq r_+/2} \left( C_0^{-2} T^{-2} |u|^2 + C_0^{-1} T^{-1} |\nabla u|^2\right)\ dx dt \\
&\quad \lesssim e^{-\alpha T r_-/2} \left( T^{-1} X + \int_{r_- \leq |x| \leq r_+} |u(0,x)|^2 |F(0,x)|\ e^{g(0,x)}\ dx \right).
\end{align*}
From \eqref{alpha-def} we have
$$ e^{-\alpha T r_-/2} = e^{- \frac{r_- r_+}{4C_0 T}}.$$
Finally, for $t=0$ and $r_- \leq |x| \leq r_+$ one has
$$ e^g \leq e^{\alpha T |x| + \frac{|x|^2}{C_0 T}} \leq e^{\frac{3 r_+^2}{2C_0 T}}$$
and
\begin{align*}
 |F| &\lesssim \alpha r_+ + \alpha T / r_- + C_0^{-1} T^{-1} + \alpha^2 T^2 + C_0^{-1} \alpha r_+ + C_0^{-2} T^{-2} r_+^2 \\
&\lesssim \frac{r_+^2}{T^2} + \frac{r_+}{r_- T} + \frac{1}{T} + \frac{r_+^2}{T^2}  + \frac{r_+^2}{T^2}  + \frac{r_+^2}{T^2} \\
&\lesssim \frac{r_+^2}{T} (T^{-1} + \frac{1}{r_+ r_-}) \\
&\lesssim \frac{r_+^2}{T} T^{-1} 
\end{align*}
since $\frac{1}{r_+ r_-} \leq \frac{1}{r_-^2} \lesssim \frac{1}{T}$ by \eqref{sigma-3}, \eqref{sigma-2}.  Bounding $\frac{r_+^2}{T} e^{\frac{3r_+^2}{2T}} \lesssim e^{\frac{2 r_+^2}{T}}$ and multiplying by $T$, we conclude that
\begin{align*}
& \int_{0}^{T/4} \int_{10r_- \leq |x| \leq r_+/2} \left( C_0^{-2} T^{-1} |u|^2 + C_0^{-1} |\nabla u|^2\right)\ dx dt \\
&\quad \lesssim e^{-\frac{r_- r_+}{4C_0 T} } \left( X + e^{2 r_+^2/T} \int_{r_- \leq |x| \leq r_+} |u(0,x)|^2\ dx \right)
\end{align*}
giving the claim.
\end{proof}

Our second application of Lemma \ref{carl} is the following quantitative version of standard parabolic unique continuation results.

\begin{proposition}[Second Carleman inequality]\label{carl-second}  Let $T, r>0$, and let ${\mathcal C}$ denote the cylindrical region
$$ {\mathcal C} \coloneqq \{ (t,x) \in \R \times \R^3: t \in [0,T]; |x| \leq r \}.$$
Let $u: {\mathcal C} \to \R^3$ be a smooth function obeying the differential inequality \eqref{lu} on ${\mathcal C}$.
Assume the inequality
\begin{align}
r^2 &\geq 4000 T. \label{sigma-5}
\end{align}
Then for any 
\begin{equation}\label{t-small}
 0 < t_1 \leq t_0 < \frac{T}{1000}
\end{equation}
one has
$$ \int_{t_0}^{2t_0} \int_{|x| \leq r/2} \left(T^{-1} |u|^2 + |\nabla u|^2\right) e^{-|x|^2/4t} \ dx dt \lesssim e^{-\frac{r^2}{500 t_0}} X + t_0^{3/2} (et_0/t_1)^{O( r^2 / t_0 )} Y
$$
where
$$
X \coloneqq \int_0^T \int_{|x| \leq r} (T^{-1} |u|^2 + |\nabla u|^2)\ dx dt
$$
and
$$ Y \coloneqq \int_{|x| \leq r} |u(0,x)|^2 t_1^{-3/2} e^{-|x|^2/4t_1}\ dx.$$
\end{proposition}

As with the previous inequality, the numerical constants here such as $1000, 500$ can be optimised if desired, but this explicit choice of constants suffices for our application.  The key feature here is the gain of $e^{-\frac{r^2}{500 t_0}}$.  Specialising to the case where $u$ vanishes to infinite order at $(0,0)$, sending $t_1 \to 0$ (which sends $(et_0/t_1)^{O( \frac{r^2}{t_0} )} Y$ to zero thanks to the infinite order vanishing), and then sending $t_0 \to 0$, we obtain a variant of a standard unique continuation theorem for backwards parabolic equations (see e.g., \cite[Theorem 4.1]{ess}).

\begin{proof}
By the pigeonhole principle, we can select a time
\begin{equation}\label{teo}
 \frac{T}{200} \leq T_0 \leq \frac{T}{100}
\end{equation}
such that
\begin{equation}\label{x-slice}
\int_{|x| \leq r} (T^{-1} |u|^2 + |\nabla u|^2)\ dx \lesssim T^{-1} X.
\end{equation}
We define
\begin{equation}\label{alpha-def2}
\alpha = \frac{r^2}{400 t_0}
\end{equation}
so from \eqref{sigma-5}, \eqref{t-small} we have
\begin{equation}\label{alpha-a}
\alpha \geq 10.
\end{equation}

We apply Lemma \ref{carl} on $[0,T_0] \times \R^3$ with the weight
$$ g \coloneqq -\frac{|x|^2}{4(t+t_1)} - \frac{3}{2} \log(t+t_1) - \alpha \log \frac{t+t_1}{T_0+t_1} + \alpha \frac{t+t_1}{T_0+t_1}$$
(which is a modification of the logarithm of the fundamental solution $\frac{1}{t^{3/2}} e^{-|x|^2/4t}$ of the heat equation) and $u$ replaced by $\psi u$, where $\psi(x)$ is a smooth cutoff supported on the region $|x| \leq r$ that equals $1$ on $|x| \leq r/2$ and obeys the estimates 
\begin{equation}\label{esti}
|\nabla^j \psi(x)| = O( r^{-j} )
\end{equation}
for $r/2 \leq |x| \leq r$ and $j=0,1,2$.  Clearly
$$ D^2 g(\nabla(\psi u), \nabla(\psi u)) = -\frac{1}{2(t+t_1)} |\nabla(\psi u)|^2.$$
We can calculate
\begin{align*}
 F &= \frac{|x|^2}{4(t+t_1)^2} - \frac{3}{2(t+t_1)} - \frac{\alpha}{t+t_1} + \frac{\alpha}{T_0+t_1} + \frac{3}{2(t+t_1)} + \left| \frac{x}{2(t+t_1)} \right|^2 \\
&= \frac{\alpha}{T_0+t_1} - \frac{\alpha}{t+t_1} 
\end{align*}
and hence
$$ LF = \frac{\alpha}{(t+t_1)^2}.$$
From Lemma \ref{carl} we thus have
\begin{align*}
& \partial_t \int_{\R^3} \left(|\nabla(\psi u)|^2 - \frac{\alpha}{2(t+t_1)} |\psi u|^2 + \frac{\alpha}{2(T_0+t_1)} |\psi u|^2\right)\ e^g dx\\
&\quad \geq \int_{\R^3} \left( \frac{\alpha}{2(t+t_1)^2} |\psi u|^2 - \frac{1}{t+t_1} |\nabla(\psi u)|^2 - \frac{1}{2} |Lu|^2 \right)\ e^g dx.
\end{align*}
To exploit this differential inequality we use the method of integrating factors.  If we introduce the energy
$$ E(t) \coloneqq \int_{\R^3} \left(|\nabla(\psi u)|^2 - \frac{\alpha}{2(t+t_1)} |\psi u|^2 + \frac{\alpha}{2(T_0+t_1)} |\psi u|^2\right)\ e^g dx$$
then we conclude from the product rule that
\begin{align*}
& \partial_t \left( \left( (t+t_1) + \frac{(t+t_1)^2}{10T_0} \right) E(t) \right) \\
&\quad \geq \left(1 + \frac{t+t_1}{5(T_0+t_1)}\right) E(t) \\
&\quad\quad + ( t+t_1 + \frac{(t+t_1)^2}{10T_0} ) \int_{\R^3} \left( \frac{\alpha}{2(t+t_1)^2} |\psi u|^2 + \frac{1}{t+t_1} |\nabla(\psi u)|^2 - \frac{1}{2} |L(\psi u)|^2 \right)\ e^g dx \\
&\quad = \int_{\R^3} \left( \frac{t+t_1}{10(T_0+t_1)} |\nabla(\psi u)|^2 + \alpha \frac{5(T_0+t_1)-(t+t_1)}{10(T_0+t_1)^2} |\psi u|^2 - \frac{1}{2} (t+t_1 + \frac{(t+t_1)^2}{10 (T_0+t_1)}) |L(\psi u)|^2 \right)\ e^g dx \\
&\quad \geq \int_{\R^3} \left( \frac{t+t_1}{10(T_0+t_1)} |\nabla(\psi u)|^2 + \frac{\alpha}{10(T_0+t_1)} |\psi u|^2 - (t+t_1) |L(\psi u)|^2 \right)\ e^g dx 
\end{align*}
and hence by the fundamental theorem of calculus
\begin{align*}
&\int_0^{T_0} \int_{\R^3} \left( \frac{t+t_1}{10(T_0+t_1)} |\nabla(\psi u)|^2 + \frac{\alpha}{10(T_0+t_1)} |\psi u|^2  \right)\ e^g dx dt \\
&\quad \leq \int_0^{T_0} \int_{\R^3} (t+t_1) |L(\psi u)|^2\ e^g dx dt + \left( (t+t_1) + \frac{(t+t_1)^2}{10(T_0+t_1)} \right) E(t)|^{t=T_0}_{t=0}.
\end{align*}
Discarding some terms, we conclude that
\begin{equation}\label{manx}
\begin{split}
&\int_0^{T_0} \int_{|x| \leq r/2} \left( \frac{t+t_1}{10(T_0+t_1)} |\nabla u|^2 + \frac{\alpha}{10(T_0+t_1)} |u|^2  \right)\ e^g dx dt \\
&\quad \leq \int_0^{T_0} \int_{\R^3} (t+t_1) |L(\psi u)|^2\ e^g dx dt \\
&\quad\quad + O\left( T_0 \int_{\R^3} |\nabla(\psi u)(T_0,x)|^2\ e^g dx + \alpha \int_{|x| \leq r} |u(0,x)|^2\ e^g dx \right).
\end{split}
\end{equation}
When $|x| \leq r/2$, one has
$$ |L(\psi u)|^2 = |Lu|^2 \leq 2T^{-2} |u|^2 + 2 T^{-1} |\nabla u|^2$$
thanks to \eqref{lu}.  By \eqref{teo}, \eqref{alpha-a} the contribution of this case is less than half of the left-hand side of \eqref{manx}.
When $r/2 \leq |x| \leq r$, we have from \eqref{esti}, \eqref{sigma-5} that
\begin{align*}
 |L(\psi u)|^2 &\lesssim T^{-2} |u|^2 + T^{-1} |\nabla u|^2 + r^{-4} |u|^2 + r^{-2} |\nabla|^2 \\
&\lesssim T^{-2} |u|^2 + T^{-1} |\nabla u|^2.
\end{align*}
Finally, $L(\psi u)$ vanishes for $|x| > r$.
Putting all this together, we conclude that
\begin{align*}
&\int_0^{T_0} \int_{|x| \leq r/2} \left( \frac{t+t_1}{T_0+t_1} |\nabla u|^2 + \frac{\alpha}{T_0+t_1} |u|^2  \right)\ e^g dx dt \\
&\quad \lesssim \int_0^{T_0} \int_{r/2 \leq |x| \leq r} (t+t_1) (T^{-2} |u|^2 + T^{-1} |\nabla u|^2)\ e^g dx dt \\
&\quad\quad + T_0 \int_{\R^3} |\nabla(\psi u)(T_0,x)|^2\ e^g dx + \alpha \int_{|x| \leq r} |u(0,x)|^2\ e^g dx.
\end{align*}
Restricting the left-hand integral to the region $t_0 \leq t \leq 2t_0$ and also bounding $t_1 \leq t_0 \leq T_0 \leq T$ in several places, we conclude that
\begin{align*}
&\int_{t_0}^{2t_0} \int_{|x| \leq r/2} \left( \frac{t_0}{T_0} |\nabla u|^2 + \frac{\alpha}{T_0} |u|^2  \right)\ e^g dx dt \\
&\quad \lesssim \int_0^{T_0} \int_{r/2 \leq |x| \leq r} (T^{-1} |u|^2 + |\nabla u|^2)\ e^g dx dt \\
&\quad\quad + T \int_{\R^3} |\nabla(\psi u)(T_0,x)|^2\ e^g dx + \alpha \int_{|x| \leq r} |u(0,x)|^2\ e^g dx.
\end{align*}
From elementary calculus we have the inequality
$$ -\frac{a}{t} - b \log t \leq b \log \frac{b}{ae}$$
for any $a,b,t > 0$ (the left-hand side attains its maximum when $t = a/b$).  When $r/2 \leq |x| \leq r$ and $0 \leq t \leq T_0$, we then have
\begin{align*}
g &\leq -\frac{|x|^2}{4(t+t_1)} - \left(\alpha+\frac{3}{2}\right) \log(t+t_1) + \alpha \log (T_0+t_1) + \alpha \\
&\leq \left(\alpha+\frac{3}{2}\right) \log \frac{4 \left(\alpha+\frac{3}{2}\right)}{e |x|^2} + \alpha \log (T_0+t_1) + \alpha \\
&\leq \left(\alpha+\frac{3}{2}\right) \log \frac{32\alpha}{e r^2} + \alpha \log (T_0+t_1) + \alpha \\
&\leq \alpha \log \frac{32 \alpha (T_0+t_1)}{r^2} + \frac{3}{2} \log \frac{32\alpha}{e r^2}  
\end{align*}
and thus by \eqref{alpha-def2}
$$ \int_0^{T_0} \int_{r/2 \leq |x| \leq r} (T^{-1} |u|^2 + |\nabla u|^2)\ e^g dx dt  \lesssim
t_0^{-3/2} \exp( \alpha \log \frac{32 \alpha (T_0+t_1)}{r^2} ) X.$$
When $|x| \leq r$ and $t = T_0$, then
$$ g \leq - \frac{3}{2} \log t_0 + \alpha$$
and $\nabla(\psi u)$ is supported on the ball $\{ |x| \leq r\}$ and obeys the estimate
$$ |\nabla(\psi u)| \lesssim |\nabla u| + r^{-1} |u| \lesssim T^{-1} |u| + |\nabla u|$$
thanks to \eqref{sigma-5}, and hence by \eqref{x-slice}
$$ T \int_{\R^3} |\nabla(\psi u)(T_0,x)|^2\ e^g dx \lesssim t_0^{-3/2} \exp(\alpha) X.$$
From \eqref{sigma-5}, \eqref{alpha-a} we have $\log \frac{32 \alpha (T_0+t_0)}{r^2} \geq 1$.  Thus
\begin{align*}
&\int_{t_0}^{2t_0} \int_{|x| \leq r/2} \left( \frac{t_0}{T_0} |\nabla u|^2 + \frac{\alpha}{T_0} |u|^2  \right)\ e^g dx dt \\
&\quad \lesssim t_0^{-3/2} \exp( \alpha \log \frac{32 \alpha (T_0+t_1)}{r^2} ) X
&\quad\quad + \alpha \int_{|x| \leq r} |u(0,x)|^2\ e^g dx.
\end{align*}
In the region $t_0 \leq t \leq 2t_0$, $|x| \leq r/2$, we have
$$
g \geq -\frac{|x|^2}{4t} - \frac{3}{2} \log(3t_0) - \alpha \log \frac{3t_0}{T_0+t_1}
$$
so that
$$ e^g \gtrsim t_0^{-3/2} e^{-|x|^2/4t} \exp\left( - \alpha \log \frac{3t_0}{T_0+t_1} \right).$$
Finally, when $t=0$ and $|x| \leq r$, we have
$$ g \leq -\frac{|x|^2}{4t_1} -\frac{3}{2} \log t_1 - \alpha \log \frac{t_1}{T_0+t_1} + \alpha $$
so that
$$ e^g \leq t_1^{-3/2} e^{-|x|^2/4t_1} \exp\left( \alpha \log \frac{e(T_0+t_1)}{t_1} \right).$$
We conclude that
\begin{align*}
&\int_{t_0}^{2t_0} \int_{|x| \leq r/2} \left( \frac{t_0}{T_0} |\nabla u|^2 + \frac{\alpha}{T_0} |u|^2  \right)\ dx dt \\
&\quad \lesssim \exp\left( \alpha \log \frac{96 \alpha t_0}{r^2} \right) X
+ \alpha \exp\left( \alpha \log(\frac{3e t_0}{t_1}) \right) t_0^{3/2} \int_{|x| \leq r} |u(0,x)|^2 t_1^{-3/2} e^{-|x|^2/4t_1}\ dx.
\end{align*}
From \eqref{alpha-def2} we have $\log \frac{96 \alpha t_0}{r^2}  \leq -1$, while from \eqref{alpha-a}, \eqref{teo}, \eqref{sigma-5}, \eqref{alpha-def2} we have $\frac{\alpha}{T_0} \gtrsim T^{-1}$ and $\frac{t_0}{T_0} \gtrsim \frac{t_0}{r^2} \gtrsim \alpha^{-1}$.  We conclude that
\begin{align*}
&\int_{t_0}^{2t_0} \int_{|x| \leq r/2} \left( |\nabla u|^2 + T^{-1} |u|^2  \right)\ dx dt \\
&\quad \lesssim \alpha^{2} e^{-\alpha} X
+ \exp\left( \alpha \log(\frac{3e t_0}{t_1}) \right) t_0^{3/2} \int_{|x| \leq r} |u(0,x)|^2 t_1^{-3/2} e^{-|x|^2/4t_1}\ dx.
\end{align*}
From \eqref{alpha-def2} we have $\alpha = O(r^2/t_0)$ and $\alpha^2 e^{-\alpha} \lesssim e^{-\frac{r^2}{500 t_0}}$, and the claim follows.
\end{proof}

\section{Main estimate}

In this section we combine the estimates in Proposition \ref{basic} with the Carleman inequalities from the previous section to obtain 

\begin{theorem}[Main estimate]\label{main-est}
Let $t_0, T, u, p, A$ obey the hypotheses of Proposition \ref{basic}, and suppose that there exists $x_0 \in \R^3$ and $N_0 > 0$ such that
$$ |P_{N_0} u(t_0,x_0)| \geq A_1^{-1} N_0$$
where as before we set $A_j \coloneqq A^{C_0^j}$.
Then
$$ T N_0^2 \leq \exp(\exp(\exp(A_6^{O(1)}))).$$
\end{theorem}

\begin{proof}  After translating in time and space we may normalise $(t_0,x_0)=(0,0)$.  Let $T_1$ be an arbitrary time scale in the interval
$$A_4 N_0^{-2} \leq T_1 \leq A_4^{-1} T.$$
By Proposition \ref{basic}(v), there exists
\begin{equation}\label{do}
 (t_1,x_1) \in [-T_1, -A_3^{-O(1)} T_1] \times B(0, A_4^{O(1)} T_1^{1/2})]
\end{equation}
and
\begin{equation}\label{nando}
N_1 = A_3^{O(1)} T_1^{-1/2}
\end{equation}
such that
$$ |P_{N_1} u(t_1,x_1)| \geq A_1^{-1} N_1.$$
From the Biot-Savart law we have
$$ P_{N_1} u(t_1,x_1) = - \Delta^{-1} P_{N_1} \nabla \times \tilde P_{N_1} \omega(t_1,x_1),$$
and hence by \eqref{local}
$$ P_{N_1} u(t_1,x_1) \lesssim N_1^{-1} \| \tilde P_{N_1} \omega(t_1) \|_{L^\infty(B(x_1,A_1/N_1))} + A_1^{-50} N_1^{-1} \| \tilde P_{N_1} \omega(t_1) \|_{L^\infty(\R^3)}.$$
From \eqref{able}, \eqref{bern} one has
$$ \| \tilde P_{N_1} \omega(t_1) \|_{L^\infty(\R^3)} \lesssim A N_1^2$$
and thus we have
$$ |\tilde P_{N_1} \omega(t_1,x'_1)| \gtrsim A_1^{-1} N_1^2$$
for some $x'_1 = x_1 + O( A_1 / N_1 ) = O( A_4^{O(1)} T_1^{1/2} )$.   By Proposition \ref{basic}(i), one has
$$ \nabla \tilde P_{N_1} \omega = O( A N_1^3 ); \quad \partial_t \tilde P_{N_1} \omega = O( A N_1^4 )$$
and thus
\begin{equation}\label{pani}
 |\tilde P_{N_1} \omega(t,x)| \gtrsim A_1^{-1} N_1^2
\end{equation}
for all $(t,x) \in [t_1, t_1 + A_1^{-2} N_1^{-2}] \times B( x'_1, A_1^{-2} N_1^{-1} )$.  By Proposition \ref{basic}(iii), there is an interval 
$$I' \subset [t_1, t_1 + A_1^{-2} N_1^{-2}] \cap [-T_1, -A_3^{-O(1)} T_1]$$ 
with $|I'| = A_3^{-O(1)} T_1$ such that
$$
 u(t,x) = O( A_3^{O(1)} T_1^{-1/2} ), \nabla u(t,x) = O( A_3^{O(1)} T_1^{-1} )
$$
and
\begin{equation}\label{omega-bound}
 \omega(t,x) = O( A_3^{O(1)} T_1^{-1} ), \nabla \omega(t,x) = O( A_3^{O(1)} T_1^{-3/2} )
\end{equation}
on $I' \times \R^3$.  By shrinking $I'$ as necessary, we may thus assume that
\begin{equation}\label{hawt}
 |u(t,x)| \leq C_0^{-1/2} |I'|^{-1/2}; \quad |\nabla u(t,x)| \leq C_0^{-1} |I'|^{-1}.
\end{equation}
On the other hand, from \eqref{pani}, \eqref{do}, \eqref{nando} one has
$$ \int_{B(0, A_4^{O(1)} T_1^{1/2})} |\tilde P_{N_1} \omega(t,x)|^2\ dx \gtrsim A_3^{-O(1)} T_1^{-1/2}$$
for all $t \in I'$.  From \eqref{local} and \eqref{hawt} this implies that
\begin{equation}\label{stip}
 \int_{B(0, A_4^{O(1)} T_1^{1/2})} |\omega(t,x)|^2\ dx \gtrsim A_3^{-O(1)} T_1^{-1/2}
\end{equation}
for all $t \in I'$.

Write $I' \coloneqq [t'-T', t']$, and let $x_* \in \R^3$ be any point with $|x_*| \geq A_5 T_1^{1/2}$.  We apply Proposition \ref{carl-second} on the slab $[0,T'] \times \R^3$ with $r \coloneqq A_5 |x_*|$, $t_0 \coloneqq T'/2$, and $t_1 \coloneqq A^{-4}_5 T'$, and $u$ replaced by the function
$$ (t,x) \mapsto \omega( t' - t, x_* + x )$$
(so that the hypothesis \eqref{lu} follows from the vorticity equation and \eqref{hawt}) to conclude that
$$ Z\lesssim \exp( - A_5 |x_*|^2 / T' ) X + (T')^{3/2} \exp( O( A_5^3 |x_*|^2 / T' ) ) Y $$
where
$$ X \coloneqq \int_{I'} \int_{B(x_*,A_5 |x_*|)} ((T')^{-1} |\omega|^2 + |\nabla \omega|^2)\ dx dt$$
and
$$ Y \coloneqq (T')^{-3/2} \int_{B(x_*,A_5 |x_*|)} |\omega(t',x)|^2 e^{-A_5^{4} |x-x_*|^2 / 4 T'}\ dx$$
and
$$ Z \coloneqq \int_{t'-T'}^{t'-T'/2} \int_{B(x_*, A_5 |x_*|/2)} (T')^{-1} |\omega|^2 e^{-|x-x_*|^2/4(t'-t)}\ dx dt.$$
From \eqref{stip} we have
$$ Z \gtrsim A_3^{-O(1)} \exp( - |x_*|^2 / 100 T' ) (T')^{-1/2}.$$
From \eqref{omega-bound} we have
$$ X \lesssim (T')^{-2} A_5^{3} |x_*|^3 \lesssim A_5^{3} \exp( |x_*|^2 / T' ) (T')^{-1/2}$$
and hence the expression $\exp( - A_5 |x_*|^2 / T' ) X$ is negligible compared to $Z$.  We conclude that
$$ Y \gtrsim \exp( - O( A_5^{3} |x_*|^2 / T' ) ) (T')^{-2}.$$
Using \eqref{omega-bound}, the contribution to $Y$ outside of the ball $B(x_*, |x_*|/2)$ is negligible, thus
$$ \int_{B(x_*,|x_*|/2)} |\omega(t',x)|^2\ dx \gtrsim \exp( - O( A_5^{3} |x_*|^2 / T' ) ) (T')^{-1/2} $$
and therefore
$$ \int_{B(0,2R) \backslash B(0,R/2)} |\omega(t',x)|^2\ dx \gtrsim \exp( - O( A_5^{3} R^2 / T' ) ) (T')^{-1/2} $$
whenever $R \geq A_5 T_1^{1/2}$.  A similar argument holds with $t'$ replaced by any time in $[t'-T'/4, t']$.  We conclude in particular that we have the Gaussian lower bound
\begin{equation}\label{gauss-lower}
\int_{-T_1}^{-A_4^{-1} T_1} \int_{B(0,2R) \backslash B(0,R/2)} |\omega(t,x)|^2\ dx dt \gtrsim \exp( -A_5^{4} R^2 / T_1 ) T_1^{1/2} 
\end{equation}
for any time scale $T_1$ and spatial scale $R$ with $A_4 N_0^{-2} \leq T_1 \leq A_4^{-1} T$ and $R \geq A_5 T_1^{1/2}$.

Now let $T_2$ be a scale for which
\begin{equation}\label{scalo}
 A_4^{2} N_0^{-2} \leq T_2 \leq A_4^{-1} T.
\end{equation}
By Proposition \ref{basic}(vi), there exists a scale
\begin{equation}\label{rb0}
 A_6 T_2^{1/2} \leq R \leq \exp(A_6^{O(1)}) T_2^{1/2} 
\end{equation}
such that on the cylindrical annulus
$$ \Omega := \{ (t,x) \in [-T_2,0] \times \R^3: R \leq |x| \leq A_6 R \}$$
one has the estimates
\begin{equation}\label{annu}
\nabla^j u(t,x) = O(A_6^{-2} T_2^{-\frac{j+1}{2}}); \quad \nabla^j \omega(t,x) = O(A_6^{-2} T_2^{-\frac{j+2}{2}})
\end{equation}
for $j=0,1$.  We apply Proposition \ref{carl-first} on the slab $[0,T_2/C_0] \times \R^3$ with $r_- \coloneqq 10 R$, $r_+ \coloneqq A_6 R / 10$, and $u$ replaced by the function
$$ (t,x) \mapsto \omega( - t, x )$$
(so that the hypothesis \eqref{lu} follows from the vorticity equation and \eqref{annu}) to conclude that
$$ Z' \lesssim \exp( - A_6^{1/2} R^2 / T_2 ) X' + \exp( \exp( A_6^{O(1)} ) ) Y'$$
where
$$ X' \coloneqq \int_{-T_2/C_0}^0 \int_{10 R \leq |x| \leq A_6 R / 10} e^{\frac{2|x|^2}{T_2}} (T_2^{-1} |\omega|^2 + |\nabla \omega|^2)\ dx dt$$
and
$$ Y' \coloneqq \int_{10 R \leq |x| \leq A_6 R / 10} |\omega(0, x)|^2\ dx$$
and
$$ Z' \coloneqq \int_{-T_2/4C_0}^0 \int_{100 R \leq |x| \leq A_1 R/20} T_2^{-1} |\omega|^2\ dx dt.$$
From \eqref{gauss-lower} (with $R$ replaced by $200R$) we have
$$ Z' \gtrsim  \exp( -A_5^{5} R^2 / T_2 ) T_2^{-1/2}.$$
Thus we either have
\begin{equation}\label{option-1}
X' \gtrsim  \exp( A_6^{1/3} R^2 / T_2 ) T_2^{-1/2}
\end{equation}
or
\begin{equation}\label{option-2}
Y' \gtrsim \exp(- \exp(A_6^{O(1)}) ) T_2^{-1/2}.
\end{equation}
Suppose for the moment that \eqref{option-1} holds.  From the pigeonhole principle, we can then find a scale
\begin{equation}\label{rb}
 10R \leq R' \leq A_6 R/10
\end{equation}
such that
$$ \int_{-T_2/C_0}^0 \int_{R' \leq |x| \leq 2R'} e^{\frac{2|x|^2}{T_2}} (T_2^{-1} |\omega|^2 + |\nabla \omega|^2)\ dx dt \gtrsim
\exp( A_6^{1/4} R^2 / T_2 ) T_2^{-1/2}$$
and thus
$$ \int_{-T_2/C_0}^0 \int_{R' \leq |x| \leq 2R'} (T_2^{-1} |\omega|^2 + |\nabla \omega|^2)\ dx dt \gtrsim
\exp( - 10 (R')^2 / T_2 ) T_2^{-1/2}$$
From \eqref{annu} we see that the contribution to the left-hand side arising from those times $t$ in the interval $[-\exp( - 20 (R')^2 / T_2 ) T_2,0]$ is negligible, thus
$$ \int_{-T_2/C_0}^{-\exp( - 20 (R')^2 / T_2 ) T_2} \int_{R' \leq |x| \leq 2R'} (T_2^{-1} |\omega|^2 + |\nabla \omega|^2)\ dx dt \gtrsim \exp( - 10 (R')^2 / T_2 ) T_2^{-1/2}.
$$
Thus by a further application of the pigeonhole principle, one can locate a time scale
\begin{equation}\label{tang}
 \exp( - 20 (R')^2 / T_2 ) T_2 \leq t_0 \leq T_2/C_0
\end{equation}
such that
$$ \int_{-2t_0}^{-t_0} \int_{R' \leq |x| \leq 2R'} (T_2^{-1} |\omega|^2 + |\nabla \omega|^2)\ dx dt \gtrsim \exp( - 10 (R')^2 / T_2 ) T_2^{-1/2}.$$
Covering the annulus $R' \leq |x| \leq 2R'$ by $O( \exp( O( (R')^2/T_2) )$ balls of radius $t_0^{1/2}$, one can then find $x_*$ with $R' \leq |x_*| \leq 2R'$ such that
\begin{equation}\label{cor}
 \int_{-2t_0}^{-t_0} \int_{B( x_*, t_0^{1/2})} (T_2^{-1} |\omega|^2 + |\nabla \omega|^2)\ dx dt \gtrsim \exp( - O( (R')^2 / T_2 ) ) T_2^{-1/2}.
\end{equation}

Now we apply Proposition \ref{carl-second} on the slab $[0,1000 t_0] \times \R^3$ with $r \coloneqq C_0^{1/4} (t_0/T_2)^{1/2} R' \leq |x_*|/10$, $t_1 \coloneqq t_0$, and $u$ replaced by the function
$$ (t,x) \mapsto \omega( -t, x_* + x )$$
(so that the hypothesis \eqref{lu} follows from the vorticity equation and \eqref{hawt}) to conclude that
\begin{equation}\label{zpp}
 Z'' \lesssim \exp( - C_0^{1/2} \frac{(R')^2}{500 T_2} ) X'' + t_0^{3/2} \exp( O( C_0^{1/2} (R')^2/T_2 ) ) Y'' 
\end{equation}
where
$$ X'' \coloneqq \int_{-T_2}^0 \int_{B(x_*,|x_*|/2)} (t_0^{-1} |\omega|^2 + |\nabla \omega|^2)\ dx dt$$
and
$$ Y'' \coloneqq t_0^{-3/2} \int_{B(x_*,|x_*|/2)} |\omega(0,x)|^2 e^{-|x-x'|^2 / 4 t_0}\ dx$$
and
$$ Z'' \coloneqq \int_{-2t_0}^{[-t_0} \int_{B(x_*, t_0^{1/2})} (t_0^{-1} |\omega|^2 + |\nabla \omega|^2) e^{-|x-x_*|^2/4|t|}\ dx dt.$$
From \eqref{cor}, \eqref{tang} one has
$$ Z'' \gtrsim \exp( - O( (R')^2 / T_2 ) ) T_2^{-1/2}.$$
From \eqref{annu}, \eqref{tang} one has
$$ X'' \lesssim T_2^{-1} t_0^{-1} (R')^3 \lesssim \exp( O( (R')^2/T_2 ) ) T_2^{-1/2}.$$
As $C_0$ is large, the first term on the right-hand side of \eqref{zpp} can thus be absorbed by the left-hand side, so we conclude that
$$ Y'' \gtrsim \exp( - O( C_0^{1/2} (R')^2 / T_2 ) ) T_2^{-2}$$
and hence
$$ \int_{R'/2 \leq |x| \leq 2R'} |\omega(0,x)|^2\ dx \gtrsim \exp( - O( C_0^{1/2} (R')^2 / T_2 ) ) T_2^{-2} t_0^{3/2}.$$
Using the bounds \eqref{rb}, \eqref{rb0}, \eqref{tang}, we conclude in particular that
\begin{equation}\label{tango}
 \int_{2 R \leq |x| \leq A_6 R / 2} |\omega(0, x)|^2\ dx \gtrsim \exp(- \exp(A_6^{O(1)}) ) T_2^{-1/2}.
\end{equation}
Note that this bound is also implied by \eqref{option-2}.  Thus we have unconditionally established \eqref{tango} for any scale $T_2$ obeying \eqref{scalo}, and for a suitable scale $R$ obeying \eqref{rb0} and the bounds \eqref{annu}.

We now convert this vorticity lower bound \eqref{tango} to a lower bound on the velocity.  The annulus $\{ 2R \leq |x| \leq A_6 R/2\}$ has volume $O( \exp( \exp( A_6^{O(1)})) T_2^{3/2} )$ by \eqref{rb0}, hence by the pigeonhole principle there exists a point $x_*$ in this annulus for which
$$ |\omega(0, x_*)| \gtrsim \exp(- \exp(A_6^{O(1)}) ) T_2^{-1}.$$
Comparing this with \eqref{annu}, we see that 
$$ |\int_{\R^3} \omega(0, x_* - r y) \varphi(y)\ dy| \gtrsim \exp(- \exp(A_6^{O(1)}) ) T_2^{-1}  $$
for some bump function $\varphi$ supported on $B(0,1)$, where $r$ is a radius of the form $r = \exp(- \exp(A_6^{O(1)}) ) T_2^{1/2}$.  Writing $\omega = \nabla \times u$ and integrating by parts, we conclude that
$$ |\int_{\R^3} u(0, x_* - r y) \nabla \times \varphi(y)\ dy| \gtrsim \exp(- \exp(A_6^{O(1)}) ) T_2^{-1/2}$$
and hence by H\"older's inequality 
$$ \int_{B(0,1)} |u(0, x_* - r y)|^3\ dy \gtrsim \exp(- \exp(A_6^{O(1)}) ) T_2^{-3/2}$$
or equivalently
$$ \int_{B(x_*, r)} |u(0,x)|^3\ dx \gtrsim \exp(- \exp(A_6^{O(1)}) ).$$
We conclude that for any scale $T_2$ obeying \eqref{scalo}, we have
$$ \int_{T_2^{1/2} \leq |x| \leq \exp( A_7) T_2^{1/2}} |u(0,x)|^3\ dx \gtrsim \exp(- \exp(A_6^{O(1)}) ).$$
Summing over a set of such scales $T_2$ increasing geometrically at ratio $\exp( A_7)$, we conclude that if
$T \geq A_4^{2} N_0^{-2}$, then
$$ \int_{\R^3} |u(0,x)|^3\ dx \gtrsim \exp(- \exp(A_6^{O(1)}) ) \log(T N_0^2).$$
Comparing this with \eqref{able}, one obtains the claim.
\end{proof}

\section{Applications}\label{final-sec}

Using the main estimate, we now prove the theorems claimed in the introduction.

We begin with Theorem \ref{main-1}.  By increasing $A$ as necessary we may assume that $A \geq C_0$, so that Theorem \ref{main-est} applies.  By rescaling it suffices to establish the claim when $t=1$, so that $T \geq 1$.  Applying Theorem \ref{main-est} in the contrapositive, we see that
\begin{equation}\label{piano}
 \| P_{N} u \|_{L^\infty_t L^\infty_x([1/2, 1] \times \R^3)} \leq A_1^{-1} N
\end{equation}
whenever $N \geq N_*$, where
$$ N_* \coloneqq \exp(\exp(\exp(A^{C_0^7}))).$$
We now insert this bound into the energy method.  As before, we split $u = u^{\lin} + u^{\nl}$ on $[1/2,1] \times \R^3$, where 
$$ u^{\lin}(t) \coloneqq e^{t\Delta} u(0)$$
and $u^{\nl} \coloneqq u - u^{\lin}$, and similarly split $\omega = \omega^{\lin} + \omega^{\nl}$.  From \eqref{bern-3}, \eqref{able} we have
\begin{equation}\label{ob}
 \| \nabla^j u^{\lin} \|_{L^\infty_t L^p_x([1/2, 1] \times \R^3)} \lesssim_j A
\end{equation}
for all $j \geq 0$ and $3 \leq p \leq \infty$.  We introduce the nonlinear enstrophy
$$ E(t) \coloneqq \frac{1}{2} \int_{\R^3} |\omega^{\nl}(t,x)|^2\ dx$$
for $t \in [1/2,1]$, and compute the time derivative $\partial_t E(t)$.  From the vorticity equation \eqref{omeganl} and integration by parts we have
\begin{equation}\label{enst}
\partial_t E(t) = -Y_1(t) + Y_2(t) + Y_3(t) + Y_4(t) + Y_5(t)
\end{equation}
where
\begin{align*}
Y_1(t) &= \int_{\R^3} |\nabla \omega^{\nl}(t,x)|^2\ dx \\
Y_2(t) &= - \int_{\R^3} \omega^{\nl} \cdot (u \cdot \nabla) \omega^{\lin}\ dx \\
Y_3(t) &= \int_{\R^3} \omega^{\nl} \cdot (\omega^{\nl} \cdot \nabla) u^{\nl}\ dx \\
Y_4(t) &= \int_{\R^3} \omega^{\nl} \cdot (\omega^{\nl} \cdot \nabla) u^{\lin}\ dx \\
Y_5(t) &= \int_{\R^3} \omega^{\nl} \cdot (\omega^{\lin} \cdot \nabla) u^{\nl}\ dx \\
Y_6(t) &= \int_{\R^3} \omega^{\nl} \cdot (\omega^{\lin} \cdot \nabla) u^{\lin}\ dx.
\end{align*}
From H\"older, \eqref{ob}, \eqref{able} we have
$$ Y_2(t), Y_6(t) \lesssim A^2 E(t)^{1/2} \lesssim A^4 + E(t)$$
and similarly
$$ Y_4(t), Y_5(t) \lesssim A E(t),$$
using Plancherel's theorem to control 
\begin{equation}\label{planch}
\| \nabla u^{\nl} \|_{L^2_x(\R^3)} \lesssim \| \omega^{\nl} \|_{L^2_x(\R^3)}.
\end{equation}
For $Y_3(t)$ we apply a Littlewood-Paley decomposition to all three factors to bound
$$ Y_3(t) \lesssim \sum_{N_1,N_2,N_3} \int_{\R^3} P_{N_1} \omega^{\nl} \cdot (P_{N_2} \omega^{\nl} \cdot \nabla) P_{N_3} u^{\nl}\ dx$$
where $N_1,N_2,N_3$ range over powers of two.  
The integral vanishes unless two of the $N_1,N_2,N_3$ are comparable to each other, and the third is less than or comparable to the other two.  Controlling the two highest frequency terms in $L^2_x$ and the lower one in $L^\infty_x$, and using the Littlewood-Paley localised version of \eqref{planch}, we conclude that
$$ Y_3(t) \lesssim \sum_{N_1,N_2,N_3: N_1 \sim N_2 \gtrsim N_3} \| P_{N_1} \omega^{\nl} \|_{L^2_x(\R^3)} \| P_{N_2} \omega^{\nl} \|_{L^2_x(\R^3)} \| P_{N_3} \omega^{\nl} \|_{L^\infty_x(\R^3)}.$$
From \eqref{able}, \eqref{bern}, the quantity $\| P_{N_3} \omega^{\nl} \|_{L^\infty_x(\R^3)}$ is bounded by $O( A N_3^2)$; for $N_3 \geq N_*$, \eqref{bern} we have the superior bound $O( A_1^{-1} N_3^2 )$.  We thus see that
$$ \sum_{N_3 \lesssim N_2} \| P_{N_3} \omega^{\nl} \|_{L^\infty_x(\R^3)} \lesssim A_1^{-1} N_2^2 + A N_*^2 $$
and thus by Cauchy-Schwarz
$$ Y_3(t) \lesssim \sum_{N_1} \| P_{N_1} \omega^{\nl} \|_{L^2_x(\R^3)}^2 (A_1^{-1} N_1^2 + A N_*^2).$$
On the other hand, from Plancherel's theorem we have
$$ Y_1(t) \sim \sum_{N_1} \| P_{N_1} \omega^{\nl} \|_{L^2_x(\R^3)}^2 N_1^2$$
and
$$ E(t) \sim \sum_{N_1} \| P_{N_1} \omega^{\nl} \|_{L^2_x(\R^3)}^2 $$
and hence
$$ Y_3(t) \lesssim A_1^{-1} Y_1(t) + A N_*^2 E(t).$$
Putting all this together, we conclude that
$$ \partial_t E(t) + Y_1(t) \lesssim A N_*^2 E(t) + A^4.$$
In particular, from Gronwall's inequality we have
$$ E(t_2) \lesssim E(t_1) + A^4 $$
whenever $1/2 \leq t_1 \leq t_2 \leq 1$ is such that $|t_2-t_1| \leq A^{-1} N_*^{-2}$.  On the other hand, from a (slightly rescaled) version of \eqref{122} we have
$$
\int_{1/2}^1 E(t)\ dt \lesssim A^4
$$
and hence on any time interval in $[1/2,1]$ of length $A^{-1} N_*^{-2}$ there is at least one time $t$ with $E(t) \lesssim A^5 N_*^2$.  We conclude that
$$ E(t) \lesssim A^5 N_*^2 \lesssim N_*^{O(1)}, $$
for all $t \in [3/4,1]$, which then also implies
$$ \int_{3/4}^1 Y_1(t) \lesssim N_*^{O(1)}.$$
  Iterating this as in the proof of Proposition \ref{basic}(iii) (or Proposition \ref{basic}(vi)), we now have the estimtes
$$ |u(t,x)|, |\nabla u(t,x)|, |\omega(t,x)|, |\nabla \omega(t,x)| \lesssim N_*^{O(1)}$$
on $[7/8,1] \times \R^3$.  This gives Theorem \ref{main-1}.

\begin{remark} More generally, one would expect in view of Theorem \ref{main-est} that any reasonable function space estimate obeyed by the linear heat equation with $L^3_x$ initial data will now also hold for classical solutions to Navier-Stokes obeying \eqref{u3}, but with an additional loss of $\exp\exp\exp(A^{O(1)})$ in the estimates.  It seems likely that a modification of the arguments above would be able to obtain such estimates, particularly if one replaces the linear estimates \eqref{very} (or \eqref{ob}) by more refined estimates that involve the profile of the initial data $u(0)$, and in particular on how the Littlewood-Paley components $\| P_N u(0) \|_{L^3_x(\R^3)}$ of the $L^3_x$ norm of that data vary with the frequency $N$.  We will not pursue this question further here.
\end{remark}

Now we prove Theorem \ref{main-2}.  We may rescale $T_* = 1$.  Let $c>0$ be a sufficiently small constant, and suppose for contradiction that
$$ \limsup_{t \to 1^+} \frac{\|u(t)\|_{L^3_x(\R^3)}}{(\log\log\log \frac{1}{1-t})^c} < +\infty,$$
thus we have
\begin{equation}\label{oil}
 \|u(t)\|_{L^3_x(\R^3)} \leq M (\log\log\log(1000 + \frac{1}{1-t}))^c
\end{equation}
for all $0 \leq t < 1$ and some constant $M$.  Applying Theorem \ref{main-1}, we obtain (for $c$ small enough) the bounds
\begin{equation}\label{water}
 \| u(t) \|_{L^\infty_x(\R^3)}, \| \nabla u(t) \|_{L^\infty_x(\R^3)}, \| \omega(t) \|_{L^\infty_x(\R^3)}, \| \nabla \omega(t) \|_{L^\infty_x(\R^3)} \lesssim_M (1-t)^{-1/10}
\end{equation}
(say) for all $1/2 \leq t < 1$.  In particular, $u$ is bounded in $L^2_t L^\infty_x$, contradicting the classical Prodi-Serrin-Ladyshenskaya blowup criterion \cite{prodi}, \cite{serrin}, \cite{lady-0}; one could also use the Beale-Kato-Majda criterion \cite{bkm} and \eqref{water} to obtain the required contradiction.  The claim follows.

\bibliographystyle{amsalpha}

\end{document}